\documentclass[11pt]{amsart}
\usepackage{amsfonts,amsmath,amssymb,amsthm,latexsym,verbatim}
\input xypic
\newtheorem{theorem}{Theorem}[section]
\newtheorem{lemma}[theorem]{Lemma}
\newtheorem{proposition}[theorem]{Proposition}
\newtheorem{corollary}[theorem]{Corollary}
\newtheorem{Problem}[theorem]{Problem}
\newtheorem{remark}[theorem]{Remark}
\theoremstyle{definition}

\numberwithin{equation}{section}

\DeclareMathOperator{\Ker}{{\mathrm{Ker}}}

\DeclareMathOperator{\image}{{\mathrm{Im}}}
\DeclareMathOperator{\soc}{{\mathrm{soc}}}

\DeclareMathOperator{\rad}{{\mathrm{rad}}}

\newcommand{\Char}{{\mathrm{char}}}
\newcommand\triv{{\bf 1}}
\newcommand\Specht{{\mathcal {S}}}

\DeclareMathOperator{\Hom}{{\mathrm{Hom}}}
\DeclareMathOperator{\End}{{\mathrm{End}}}
\DeclareMathOperator{\Ind}{{\mathrm{Ind}}}
\DeclareMathOperator{\Res}{{\mathrm{Res}}}

\DeclareMathOperator{\IBR}{{\mathrm{IBr}}}
\DeclareMathOperator{\Sym}{{\mathrm{Sym}}}
\DeclareMathOperator{\Alt}{{\mathrm{Alt}}}

\newcommand{\CL}{{\mathcal C}}
\newcommand{\SCL}{{\mathcal S}}
\newcommand{\F}{{\mathbb {F}}} 
 
\newcommand{\Z}{{\mathbb {Z}}} 
 
\newcommand{\C}{{\mathbb {C}}}
\newcommand{\K}{{\mathbb {K}}}
\newcommand{\MB}{\mathbf{M}} 
\newcommand{\bfr}{\mathfrak{b}} 
\newcommand{\la}{\langle} 
\newcommand{\ra}{\rangle}
\newcommand{\SSS}{{\mathsf {S}}}
\newcommand{\AAA}{{\mathsf {A}}}
\DeclareMathOperator{\rank}{{\mathrm {rank}}} 	

\DeclareMathOperator{\Stab}{{\mathrm {Stab}}} 
\DeclareMathOperator{\sign}{{\mathrm {sign}}}
\DeclareMathOperator{\sgn}{{\mathrm {sgn}}} 	

\newdimen\hoogte    \hoogte=12pt
\newdimen\breedte   \breedte=14pt
\newdimen\dikte     \dikte=0.5pt

\newenvironment{Young}{\begingroup
       \def\vr{\vrule height0.89\hoogte width\dikte depth 0.2\hoogte}
       \def\fbox##1{\vbox{\offinterlineskip
                    \hrule height\dikte
                    \hbox to \breedte{\vr\hfill##1\hfill\vr}
                    \hrule height\dikte}}
       \vbox\bgroup \offinterlineskip \tabskip=-\dikte \lineskip=-\dikte
            \halign\bgroup &\fbox{##\unskip}\unskip  \crcr }
       {\egroup\egroup\endgroup}
\def\diagram#1{\relax\ifmmode\vcenter{\,\begin{Young}#1\end{Young}\,}\else%
              $\vcenter{\,\begin{Young}#1\end{Young}\,}$\fi}

\begin{document}
\title[Representations of the alternating group]
{Representations of the alternating group which are irreducible over subgroups. 
II}
\author{\sc Alexander Kleshchev}
\address
{Department of Mathematics\\ University of Oregon\\
Eugene\\ OR~97403, USA}
\email{klesh@uoregon.edu}
\author{\sc Peter Sin}
\address{Department of Mathematics\\University of Florida\\ 
Gainesville
FL 32611-8105\\ USA}
\email{sin@ufl.edu}
\author{\sc Pham Huu Tiep}
\address
{Department of Mathematics\\
University of Arizona\\ Tucson\\ AZ~85721, USA} 
\email{tiep@math.arizona.edu}
\date{}
\thanks{
\indent The first author was supported by the NSF (grant DMS-1161094) and the Humboldt Foundation. The second author was partially supported by the Simons Foundation 
(grant \#204181). The third author was supported by the NSF 
(grants DMS-0901241 and DMS-1201374). 
}


\begin{abstract}
We prove that non-trivial representations of the alternating group 
$\AAA_n$ are reducible over a primitive proper subgroup which is isomorphic to some alternating group $\AAA_m$. A similar result is established for finite simple classical groups embedded in
$\AAA_n$ via their standard rank $3$ permutation representations. 
\end{abstract}

\maketitle
\section{Introduction}
If $\Gamma$ is a transitive permutation group with a point stabilizer $X$ then $\Gamma$ is primitive if and only if $X<\Gamma$ is a maximal subgroup. So studying primitive permutation groups is equivalent to studying maximal subgroups.
In most problems
involving a finite primitive group $\Gamma$, the Aschbacher-O'Nan-Scott theorem \cite{AS} allows one to concentrate on the 
case where $\Gamma$ is almost quasi-simple, i.e. $L\lhd \Gamma/Z(\Gamma)\leq \operatorname{Aut}(L)$ for a non-abelian simple group $L$. 
The results of 
Liebeck-Praeger-Saxl \cite{LPS} and Liebeck-Seitz \cite{LiS} then allow one to 
assume furthermore that $\Gamma$ is a finite classical group. 

In the latter case, the
possible structure of maximal subgroups $X$ is described by 
Aschbacher's theorem \cite{A}: if $X<\Gamma$ is maximal then $X$ belongs to
\begin{equation}\label{EM}
\SCL\cup \bigcup^{8}_{i=1}\CL_{i},
\end{equation}
where $\CL_{1},\dots,\CL_{8}$ are
collections of certain explicit natural subgroups of $\Gamma$, and $\SCL$ is the collection
of almost quasi-simple groups that act absolutely irreducibly on the natural module for the classical group $\Gamma$. 

It is not true, however, that every subgroup $X$ in (\ref{EM}) is actually maximal in
$\Gamma$. For $X \in \cup^{8}_{i=1}\CL_{i}$, the maximality of $X$ has been determined by
Kleidman-Liebeck \cite{KL}. 
So let $X \in \SCL$. If $X$ is {\em not} maximal then $X < G < \Gamma$ for a certain maximal subgroup  $G$ in $\Gamma$. The most challenging case to handle is when $G \in \SCL$ as well. This motivates the following problem, where $\F$ is an algebraically closed
field of characteristic $p \geq 0$: 

\begin{Problem}\label{restr}
Classify all triples $(G,V,X)$ where 
$G$ is an almost quasi-simple finite group, $V$ 
is an $\F G$-module of dimension greater than one, and $X$ is a proper subgroup of $G$ such that 
the restriction $\Res^G_X V$ is irreducible.
\end{Problem}

Many results have been obtained concerning various cases of Problem
\ref{restr} --- see for instance  \cite{KT2} and references therein. In this paper, we are mostly
interested in the case $G$ is is the alternating group $G=\AAA_n$ or the symmetric group $\SSS_n$. In this case, under the 
assumption $p > 3$ (or $p = 0$), Problem \ref{restr} has been solved in \cite{BK,KS2} --- see also \cite{KT1} for double-covers $\hat\AAA_n$ and $\hat\SSS_n$ and \cite{Saxl,KW} for the characteristic zero case. 
A number of techniques employed in these papers unfortunately break down in 
the case $(G,X) = (\AAA_n,\AAA_m)$ and $p = 2,3$ (and especially when
$X$ is a primitive subgroup of $G$). On the other hand, this case is
of crucial importance in a number of applications. The purpose of this paper
is to resolve this important case,  and our main result is:

\begin{theorem}\label{main}
Let $X \cong \AAA_m$ be a primitive subgroup of $\AAA_n$ with $n > m \geq 9$.
Let $\F$ be an algebraically closed field of arbitrary characteristic and  $V$ be a non-trivial $\F\AAA_n$-module. Then $V$ is reducible over $X$.
\end{theorem}

The bound $m\geq 9$ is the best possible --- see Remark~\ref{RSmallCases} and Lemma \ref{small}. 

We emphasize that our methods also apply to many other primitive subgroups of $\AAA_n$. To illustrate this, in this paper we handle the simple classical groups $X$ that embed in $\AAA_n$ via their standard rank $3$ permutation representations:
\begin{theorem}\label{r3-main}
Let $X = L/Z(L)$ be a finite simple classical group, where $L$ is one of the following 
group: $SL_d(q)$, $SU_d(q)$, or $Sp_d(q)'$ with $d \geq 4$, and $\Omega^\pm_d(q)$ with 
$d \geq 5$. Let $W$ denote the natural $d$-dimensional module for $L$, and let $X$ be embedded
in $\Sym(\Omega) = \SSS_n$ via its rank $3$ permutation action on the set $\Omega$ of 
$2$-dimensional subspaces of $W$ in the case $L = SL_d(q)$, and of $1$-dimensional singular 
subspaces of $W$ otherwise. If $V$ is any $\F\AAA_n$-module of dimension $>1$, then 
$\Res^{\AAA_n}_{X}V$ is reducible. 
\end{theorem}
We plan to extend  this result to the remaining simple primitive subgroups of $\AAA_n$ in a sequel. 
Together with the results of \cite{BK, KS2} and the current paper, this will completely solve 
Problem \ref{restr} for $G = \AAA_n$ in the cases that are of most interest for the 
Aschbacher-Scott program.

\smallskip
The paper is organized as follows. Basic notions are recalled in \S2. 
Theorem \ref{hom} in \S3 compares the dimensions of the Hom-spaces of irreducible 
$\SSS_n$-modules in characteristic $2$ over certain Young subgroups of $\SSS_n$
when $n$ is even. 
Then Propositions \ref{ext1}, \ref{ext2}, and  \ref{ext3} in \S4   
show in particular that the $p$-modular irreducible representations of $\AAA_n$ which
do not extend to $\SSS_n$ must have large enough dimension (at least exponential
in $n$). These results, which we believe are also of independent interest, allow
us to discard non-$\SSS_n$-extendible $\AAA_n$-modules in the proof
of Theorem \ref{main}. In \S5
we describe the submodule structure of the permutation modules of $\SSS_n$ 
acting on subsets of $\{1,2, \ldots ,n\}$ of cardinality $2$ or $3$ in characteristic $2$,
again in the case of even $n$. This description plays a key role in the proof of 
Theorem \ref{reduction} in \S6, which gives a criterion for a
$2$-modular irreducible $\SSS_n$-representation to be reducible over certain subgroups of
$\SSS_n$.  Theorem \ref{reduction} is then used in \S7 to show that non-trivial
$2$-modular irreducible $\AAA_n$-representations are reducible 
over $\AAA_m$, if $\AAA_m$ 
is embedded into $\AAA_n$ via its actions on subsets or set partitions of 
$\{1,2, \ldots ,m\}$ --- see Theorem \ref{main-alt1}. Theorem \ref{main} is proved
in \S8, which also contains further results concerning non-primitive embeddings of
$\AAA_m$ into $\AAA_n$.  The final \S9 is devoted to the proof of Theorem \ref{r3-main}.

\section{Preliminaries}\label{SPrel}
Throughout the paper, unless otherwise stated, we assume that the ground field $\F$ is algebraically closed, and  $p:=\Char(\F)$. For a group $G$, the trivial $\F G$-module is denoted $\triv_G$ or simply $\triv$ if it is clear what $G$ is.  If $V$ is an $\F G$-module, we denote by $\soc(V)$ the socle of $V$, and for $n=1,2,\dots$, define $\soc^n(V)$ from $\soc^1(V)=\soc(V)$ and $\soc^n(V)/\soc^{n-1}(V)=\soc(V/\soc^{n-1} (V))$ for $n>1$. We refer to the quotients 
$\soc^n(V)/\soc^{n-1}(V)$ as the {\em socle layers} of $V$ and usually list them from bottom to top, i.e. first $\soc(V)$, then $\soc^2(V)/\soc(V)$, etc.

For $n\in\Z_{>0}$, let 
$$\Omega := \{1,2,\ldots, n\}.$$
For $r=1,\dots,n$, denote by $\Omega_r$ the set of $r$-element subsets of $\Omega$. 
The symmetric group $\SSS_n$ acts naturally on the sets $\Omega=\Omega_1,\Omega_2,\dots,\Omega_n$ and the stabilizer of an element of $\Omega_r$ is conjugate to the subgroup
$\SSS_{n-r,r}:=\SSS_{\{1,2,\ldots,n-r\}}\times \SSS_{\{n-r+1,\ldots,n\}}$.
We write  $\SSS_{n-1,1}$ simply as  $\SSS_{n-1}$. 

We denote by 
$$M_r=\F\Omega_r \cong \Ind^{\SSS_n}_{\SSS_{n-r,r}} \triv_{\SSS_{n-r,r}}
\qquad(1\leq r\leq n)
$$ 
the permutation module for the action of $\SSS_n$ on $\Omega_r$. 

We recall some basic notions of representation theory of symmetric groups referring to \cite{JamesBook} for details. The irreducible $\F \SSS_n$-modules are labeled by $p$-regular partitions of $n$ (if $p=0$ then $p$-regular partitions are interpreted as all partitions). If $\lambda$ is a $p$-regular partition of $n$, the corresponding irreducible module is denoted $D^\lambda$. The Specht modules over $\F\SSS_n$ are labeled by  partitions of $n$. If $\lambda$ is such a partition, the corresponding Specht module is denoted $S^\lambda$. 

Let $p=2$. Consider the partition 
$$
\alpha_n=
\left\{
\begin{array}{ll}
(k+1,k-1) &\hbox{if $n=2k$ is even,}\\
(k+1,k) &\hbox{if $n=2k+1$ is odd.}
\end{array}
\right.
$$
The irreducible module $D^{\alpha_n}$ is called {\em the basic spin module} for $\SSS_n$. It is known \cite[Table III]{Wales} that 
$$
\dim D^{\alpha_n}=2^{\lfloor (n-1)/2\rfloor}.
$$

Let $\sgn_n$ be the sign module over $\F\SSS_n$. For any $p$-regular partition $\lambda$, we have that $D^\lambda\otimes \sgn_n$ is an irreducible $\F\SSS_n$-module, so we can write
$
D^\lambda\otimes \sgn_n\cong D^{\lambda^\MB},
$
where
$$\MB:\lambda \mapsto \lambda^{\MB}$$ 
is the {\it Mullineux involution}  
on the set of $p$-regular partitions
of $n$. To describe the Mullineux involution, we briefly recall the notion of the {\it Mullineux symbol} $G(\lambda)$ of 
$\lambda$, referring the reader to \cite{FK} for details. Let $h_1$ be the number of nodes in the 
$p$-rim of $\lambda$, and let $r_1$ be the number of rows in $\lambda$. Delete the 
$p$-rim and repeat to obtain sequences $h_1, h_2, \ldots $ and $r_1, r_2, \ldots$.
Let $k$ be such that $h_{k+1} = r_{k+1} = 0$ but $h_k \neq 0 \neq r_k$. Then 
$$G(\lambda) := \begin{pmatrix}h_1 & h_2 & \ldots & h_k\\ 
  r_1 & r_2 & \ldots & r_k \end{pmatrix}.$$                   
It was proved in \cite{Mu} that $\lambda$ is uniquely determined by 
$G(\lambda)$. Moreover, we have 
$$G(\lambda^\MB) = \begin{pmatrix}h_1 & h_2 & \ldots & h_k\\ 
  h_1-r_1+\epsilon_1 & h_2-r_2+\epsilon_2 & \ldots & h_k-r_k+\epsilon_k
  \end{pmatrix},$$
where $\epsilon_i := 0$ if $p|h_i$ and $\epsilon_i := 1$ otherwise.
This description of $\MB$ is the main result of \cite{FK} (see also \cite{BOMull}),  which
was conjectured by Mullineux. 

Given an irreducible representation $D^\lambda$, either the restriction 
$E^\lambda:= \Res^{\SSS_n}_{\AAA_n}D^\lambda$ is irreducible or 
$\Res^{\SSS_n}_{\AAA_n}D^\lambda \cong E^\lambda_+\oplus E^\lambda_-$, a direct sum of two inequivalent irreducible representations. Moreover, every irreducible $\F\AAA_n$-module is isomorphic to one of $E^\lambda_{(\pm)}$, and the only non-trivial isomorphism of the form $E^\lambda_{(\pm)}\cong E^\mu_{(\pm)}$ is $E^\lambda\cong E^{\lambda^\MB}$. 

If $p\neq 2$, then $\Res^{\SSS_n}_{\AAA_n}D^\lambda$ is reducible if and only if $\lambda=\lambda^\MB$. If $p=2$, then an explicit criterion for reducibility of 
$\Res^{\SSS_n}_{\AAA_n}D^{\lambda}$ is given in \cite[Theorem 1.1]{B}.

\section{Comparing some Hom-spaces}
Throughout this section we assume that $p=2$. In this section we get some results on the dimensions 
\begin{equation}\label{EDR}
d_r(V):=\dim \Hom_{\F \SSS_n}(M_r,
\End_\F(V))=\dim\End_{\F\SSS_{n-r,r}}(\Res^{\SSS_n}_{\SSS_{n-r,r}} V).
\end{equation}
The last equality follows using $M_r=\Ind^{\SSS_n}_{\SSS_{n-r,r}}\triv_{\SSS_{n-r,r}}$ and Frobenius reciprocity. 

\begin{lemma} \label{LIrrMr}
Let $V$ be an irreducible $\F\SSS_n$-module and $1\leq r\leq n$. Then $d_r(V)=1$ if and only
if $\Res^{\SSS_n}_{\SSS_{n-r,r}} V$ is irreducible. 
\end{lemma}

\begin{proof}
The sufficiency of the condition is clear. Conversely, irreducible $\F\SSS_n$-modules are self-dual, so the restriction $\Res^{\SSS_n}_{\SSS_{n-r,r}} V$ is self-dual. Since irreducible  $\F\SSS_{n-r,r}$-modules are also self-dual, the head of $\Res^{\SSS_n}_{\SSS_{n-r,r}} V$ is isomorphic to its socle. So if $\Res^{\SSS_n}_{\SSS_{n-r,r}} V$ is reducible then $d_r(V)>1$. 
\end{proof}
 
The goal of this section is to prove the following result:

\begin{theorem}\label{hom}
Let $V$ be a simple $\F \SSS_n$-module and $2|n \geq 6$. Then one of the 
following statements holds:
\begin{enumerate}
\item[{\rm (i)}] $d_3(V) > d_1(V)$. 
\item[{\rm (ii)}] $V\cong D^{\alpha_n}$ is the basic spin module 
or $V\cong\triv$ is the trivial module, in which cases we have $d_3(V)=d_1(V)$. 
\end{enumerate}
\end{theorem}

Let $V = D^\lambda$  
for a $2$-regular partition 
$$\lambda = (\lambda_1 > \ldots > \lambda_s > 0)$$ 
of $n$.
If $s = 1$, then $V$ is the trivial module, and Theorem \ref{hom} holds 
trivially.

\begin{lemma}\label{hom1}
Theorem \ref{hom} holds if $d_1(V) = 1$.
\end{lemma}

\begin{proof}
Suppose that $d_1(V)=d_3(V) = 1$ and $\dim V > 1$. 
By Lemma~\ref{LIrrMr}, $V$ is irreducible over 
$\SSS_{n-1}$ and $\SSS_{3,n-3}$, and so the lemma follows from \cite[Theorem 10]{Phillips}. 
\end{proof}

\begin{lemma}\label{hom2}
Theorem \ref{hom} holds if $s=2$.
\end{lemma}
\begin{proof}
Since $n$ is even, we have that the restriction $D^\lambda{|}_{\SSS_{n-1}}$ is irreducible by \cite{KRes}. So $d_1(V)=1$, and we may apply Lemma~\ref{hom1}. 
\end{proof}

\begin{lemma} \label{LX} 
Let 
$$
x:=\sum_{g\in\SSS_{\{1,2,3\}},\ \sigma\in \SSS_{\{1,4\}}\times \SSS_{\{2,5\}}\times\SSS_{\{3,6\}}}(\sign\sigma) \sigma g\sigma^{-1}\in\F\SSS_n.
$$
If $s>2$, then $xD^\lambda\neq 0$. 
\end{lemma}
\begin{proof}
By \cite[Lemma 4.7]{BaK} with 
$(m,p) = (3,2)$, we see that the restriction $V{|}_{\SSS_6}$ has a composition factor 
isomorphic to $D^{(3,2,1)}$, which for $p=2$ is isomorphic to the Specht module $S^{(3,2,1)}$. Since $x\in\F\SSS_6$, it suffices to prove that $x S^{(3,2,1)}\neq 0$. We use the notation of \cite[\S4]{JamesBook}; in particular, $e_t$ is the polytabloid
and $\{t\}$ is the tabloid corresponding to a $(3,2,1)$-tableau $t$. 
Let 
$$
t:=\diagram{$1$&$4$&$6$ \cr $2$ & $5$\cr$3$\cr}\, ,\quad 
s:=\diagram{$1$&$2$&$4$ \cr $3$ & $5$\cr$6$\cr}.
$$
An explicit calculation shows that $\{s\}$ appears in $xe_t$ with coefficient $1$. 
\end{proof}

We now complete the proof of Theorem~\ref{hom}. Recall that $M_3$ is the permutation module on all three element subsets $\{i,j,k\}\subseteq \Omega$, while $M_1$ is the permutation module on the one element subsets $\{i\}\subseteq \Omega$. Consider the $\F\SSS_n$-module homomorphism 
$$
f:M_3\to M_1,\ \{i,j,k\}\mapsto  \{i\}+\{j\}+\{k\}.
$$
It is easy to see that $f$ is surjective. So it induces an injective linear map 
$$
f^*:\Hom_{\F \SSS_n}(M_1, \End_\F(V))\to \Hom_{\F \SSS_n}(M_3, \End_\F(V)),\  \psi\mapsto \psi\circ f.
$$
It suffices to prove that $f^*$ is not surjective. 

We exhibit an element $\phi\in \Hom_{\F \SSS_n}(M_3, \End_\F(V))$ which is not in the image of $f^*$. For $\Theta\subseteq \Omega$ let $\SSS_\Theta\subseteq \SSS_n=\SSS_\Omega$ be the subgroup of all permutations which stabilize the elements of $\Omega\setminus \Theta$. Now, define $\phi$ as follows: 
\begin{equation}\label{EPhi}
\phi(\{i,j,k\})(v):=\sum_{g\in \SSS_{\{i,j,k\}}} gv, \qquad(\{i,j,k\}\subseteq\Omega,\ v\in V).
\end{equation}

If $\phi\in\image f^*$, then  $\phi=\psi\circ f$ for some $\psi\in  \Hom_{\F \SSS_n}(M_1, \End_\F(V))$. Consider the element 
$$
E=\sum_{\sigma\in \SSS_{\{1,4\}}\times \SSS_{\{2,5\}}\times\SSS_{\{3,6\}}}(\sign\sigma)\sigma\{1,2,3\}\in M_3.
$$
Note that $f(E)=0$. So $\phi(E)=\psi( f(E))=0$. On the other hand, we compute $\phi(E)$ using (\ref{EPhi}):
$$
\phi(E)(v)=\sum_{\sigma\in \SSS_{\{1,4\}}\times \SSS_{\{2,5\}}\times\SSS_{\{3,6\}}}(\sign\sigma)\sum_{g\in\SSS_{\{\sigma(1),\sigma(2),\sigma(3)\}}}gv
= xv,
$$ 
where $x$ is as in Lemma~\ref{LX}. Now Lemma~\ref{LX} yields a contradiction. 

\section{Dimension and extendibility to $\SSS_n$}
First we prove the following statement, which relies on some results of
\cite{B} and \cite{GLT}:

\begin{proposition}\label{ext1}
Let $p=2$, $n \geq 5$, and 
let $V$ be an irreducible $\F\AAA_n$-module. Suppose that $V$ does not extend to
$\SSS_n$. Then $\dim V \geq 2^{(n-6)/4}$.
\end{proposition}

\begin{proof}
By assumption, $W := \Ind^{\SSS_n}_{\AAA_n}(V)$ is an irreducible 
$\F\SSS_n$-module, and $\dim W=2  \dim V$. Let 
$\lambda = (\lambda_1 > \lambda_2 > \ldots \lambda_s > 0)$ be the partition of $n$ 
into distinct parts corresponding to $W$. Since $W$ is reducible over $\AAA_n$,
by \cite[Theorem 1.1]{B}, we have $s > 1$ and $\lambda_1-\lambda_2 \in \{1,2\}$. 
In particular, 
$n \geq \lambda_1 + \lambda_2 \geq 2\lambda_1-2,$ 
i.e. $\lambda_1 \leq (n+2)/2$. Now 
$$\dim W \geq 2^{\frac{n-\lambda_1}{2}} \geq 2^{\frac{n-(n+2)/2}{2}} = 2^{\frac{n-2}{4}},$$
thanks to \cite[Theorem 5.1]{GLT}. 
\end{proof}

We will also need the following branching result which is of interest in its own right:

\begin{proposition}\label{two}
Let $p=2$, and $\lambda = (\lambda_1 > \ldots > \lambda_s > 0)\neq (n)$ be a non-trivial
$2$-regular partition of 
$n$. If $2\lambda_1-n \geq k \geq 3$  then  
the restriction of $D^\lambda$ to a natural subgroup $\SSS_k$ of $\SSS_n$ 
affords both $\triv=D^{(k)}$ and $D^{(k-1,1)}$ as composition factors. 
\end{proposition}

\begin{proof}
We apply induction on $m := n-\lambda_1 \geq 1$. If $m = 1$, 
then $D^\lambda = D^{(n-1,1)}$ is the heart of the natural permutation module, and the 
statement follows easily. Let $m \geq 2$. 

\smallskip
Case 1: $\lambda_1-\lambda_2$ is odd. Then, in the terminology of \cite[Definition 0.3]{K2}, $2$ is
a normal index. Let $j \geq 2$ be the largest normal index; in particular
$j$ is a good index in the sense of \cite[Definition 0.3]{K2} again. Then by \cite[Theorem 0.5]{K2}, $D^\mu$ is a simple submodule 
of $D^\lambda|_{\SSS_{n-1}}$, where 
$$\mu = \lambda(j) := 
  (\lambda_1, \ldots ,\lambda_{j-1},\lambda_j-1,\lambda_{j+1}, \ldots ,\lambda_s)
  \vdash (n-1).$$
Note that $2\lambda_1 - (n-1)\geq k+1$ and $(n-1)-\lambda_1 = m-1$. Hence 
we can apply the induction hypothesis to $D^\mu$ restricted to $\SSS_k$.

\smallskip
Case 2: $\lambda_1-\lambda_2$ is even. Now $1$ is
a normal index. Then by \cite[Theorem 0.4]{K2}, $D^\nu$ is a composition factor 
of $D^\lambda|_{\SSS_{n-1}}$, where 
$$\nu = \lambda(1) := 
  (\lambda_1-1, \lambda_2, \ldots ,\lambda_s) \vdash (n-1).$$
Since $(\lambda_1-1)-\lambda_2$ is odd, as in Case 1 we now see that $2$ is
a normal index for $\nu$. Let $j \geq 2$ be the highest normal index of $\nu$; 
in particular $j$ is a good index. Then again by \cite[Theorem 0.5]{K2}, 
$D^\mu$ is a simple submodule of $D^\lambda|_{\SSS_{n-2}}$, where 
$$\mu = \nu(j) := 
  (\lambda_1-1, \ldots ,\lambda_{j-1},\lambda_j-1,\lambda_{j+1}, \ldots ,\lambda_s)
  \vdash (n-2).$$
Note that $2(\lambda_1-1) - (n-2) \geq k$ and $(n-2)-(\lambda_1-1) = m-1$. Hence 
we can apply the induction hypothesis to $D^\mu$ restricted to $\SSS_k$.  
\end{proof}

Using the Mullineux involution, we prove an analogue of 
Proposition \ref{ext1} for $p\neq 2$ (certainly, the most interesting case being 
$p \leq n$): 

\begin{proposition}\label{ext2}
Let $n \geq 5$ and $p \neq  2$.
\begin{enumerate}
\item[{\rm (i)}] Let $\lambda = (\lambda_1, \lambda_2, \ldots )$
be a $p$-regular partition of $n$. Suppose that  $\lambda_1 \geq (n+p+2)/2$.
Then $D^\lambda$ is irreducible over $\AAA_n$. 
\item[{\rm (ii)}] Let $V$ be an irreducible $\F\AAA_n$-module. Suppose that $V$ does 
not extend to $\SSS_n$. Then $\dim V \geq 2^{(n-p-5)/4}$.
\end{enumerate}
\end{proposition} 

\begin{proof}
(i) Recalling the definition of the Mullineux map from Section~\ref{SPrel}, denote the partitions obtained from $\lambda$ by successively removing $p$-rims
as $\lambda^{(j)}$, $1 \leq j \leq k$. We prove the statement by induction on 
$n - \lambda_1$. Since $\lambda_1 \geq (n-\lambda_1) +p+2$ by the assumption, 
the first $p$-segment of the $p$-rim of $\lambda$ has length $p$. Assume for a contradiction 
that $\lambda = \lambda^\MB$. Then  
\begin{equation}\label{rim1}
  h_i+\epsilon_i = 2r_i
\end{equation}   
for $1 \leq i \leq k$.

\smallskip
Suppose first that the $p$-rim of $\lambda' := (\lambda_2, \lambda_3, \ldots )$ has at most
$p-1$ nodes. Write $h_1 = p+x$ and $r_1 = 1+y$, where $0 \leq x \leq p-1$ and $y$ is
the number of rows of $\lambda'$. Then according to (\ref{rim1}) we have 
\begin{equation}\label{rim2}
  p+x  \leq p+x +\epsilon_1 = h_1 + \epsilon_1 = 2r_1 = 2 +2y.
\end{equation}  
Note that $x$ is the length of the $p$-rim of $\lambda'$. Hence $y \leq x$, and so
(\ref{rim2}) yields $x \geq p-2 > 0$. In turn, this implies that $p{\not{|}}h_1$, whence
$\epsilon_1  = 1$ and (\ref{rim2}) yields $x = y = p-1$ (as $x \leq p-1$). Recall we 
are assuming that the $p$-rim of $\lambda'$ has at most $p-1$ nodes. It follows that
the $p$-rim of $\lambda'$ has exactly $p-1$ nodes and $\lambda'$ also has $p-1$ rows.
This can happen only when $\lambda' = (1^{p-1})$, a column of $p-1$ nodes. 
In this case, $\lambda^{(1)} = (\lambda_1-p)$
has one part, which is of length $\geq 2$. Hence the $p$-rim of 
$\lambda^{(1)}$ is of length $p$ (if $\lambda_1 \geq 2p$), or $z \geq 2$
(where $2p-1 \geq \lambda_1 = p+z \geq p+2$). Correspondingly, $r_2 = 1$ and 
$(h_2, \epsilon_2) = (p,0)$ or $(z,1)$. In either case
$$h_2 + \epsilon_2 \geq z+1 \geq 3 > 2r_2,$$
contrary to (\ref{rim1}).

\smallskip
Assume now that the $p$-rim of $\lambda'$ has at least $p$ nodes. Then, aside from 
the first $p$-segment contained in the first row, the $p$-rim of $\lambda$ contains
at least $p$ nodes of $\lambda'$. It follows that the condition
$\lambda_1 \geq (n-\lambda_1)+p+2$ also holds for $\lambda^{(1)}$. By the induction
hypothesis, $\lambda^{(1)}$ is not equal to its Mullineux dual, i.e. 
$h_i - r_i + \epsilon_i \neq r_i$ for some $i \geq 2$, again contradicting (\ref{rim1}).
 
 \smallskip
(ii) By assumption, $W := \Ind^{\SSS_n}_{\AAA_n}(V)$ is an irreducible 
$\F\SSS_n$-module and  $\dim W=2 \dim V$. Let 
$\lambda = (\lambda_1, \lambda_2, \ldots )$ be the $p$-regular partition of $n$ 
corresponding to $W$. Since $W$ is reducible over $\AAA_n$, $\lambda_1 \leq (n+p+1)/2$
by (i). It now follows by \cite[Theorem 5.1]{GLT} that
$$\dim W \geq 2^{\frac{n-\lambda_1}{2}} \geq 2^{\frac{n-(n+p+1)/2}{2}} = 2^{\frac{n-p-1}{4}},$$
which implies the result. 
\end{proof}

Here is another version of Proposition \ref{ext2}: 

\begin{proposition}\label{ext3}
Let $p > 2$, $n \geq 5$, 
and $\lambda = (\lambda_1, \lambda_2, \ldots )$
be a $p$-regular partition of $n$. Suppose that there is some $s \geq 1$ such that
$$\lambda_1-\lambda_2 \geq \lambda_2-\lambda_3 \geq \ldots 
   \geq \lambda_s-\lambda_{s+1} \geq p$$
and 
$$\sum^{s}_{i=1}\lfloor \frac{\lambda_i - \lambda_{i+1}}{p} \rfloor >  \frac{n}{2p-1}.$$
Then $D^\lambda$ is irreducible over $\AAA_n$. 
In particular, if
$$\lambda_1 \geq \lambda_2 + p  \left\lceil \frac{n+1}{2p-1} \right\rceil,$$
then $D^\lambda$ is irreducible over $\AAA_n$. 
\end{proposition} 

\begin{proof}
Assume that $D^\lambda$ is reducible over $\AAA_n$. Then $\lambda = \lambda^\MB$ 
and so 
\begin{equation}\label{hr}
  \sum^k_{i=1}(h_i-2r_i+\epsilon_i) = 0.
\end{equation}
We will estimate $h_1-2r_1+\epsilon_1$ by 
going down the $p$-segments of the $p$-rim of $\lambda$. Since 
$\lambda_1-\lambda_2 \geq p$, the first $p$-segment consists of $p$ nodes
of the first row and so contributes $p-2$ to $h_1-2r_1+\epsilon_1$. More 
generally, any horizontal $p$-segment of length $p$ contributes $p-2$ to 
$h_1-2r_1+\epsilon_1$. On the other hand, since $\lambda$ is $p$-regular, any 
non-horizontal $p$-segment of length $p$ has height $\leq (p-1)$ and so it 
contributes at least $p-2(p-1) = 2-p$ to $h_1-2r_1+\epsilon_1$.
Suppose the $p$-rim also has a $p$-segment of length $j$ less than $p$. Then it
must be the last $p$-segment, and $\epsilon_1 = 1$. So the contribution of
this $p$-segment to $h_1-2r_1+\epsilon_1$ is $\geq j-2j+1 = 1-j$.

As the $p$-rims are removed in succession, 
let $a$ be the total number of horizontal $p$-segments of length $p$,
$b_p$ be the total number of non-horizontal $p$-segments of length $p$, and 
$b_j$ be the total number of $p$-segments of length $1 \leq j < p$, so
that $n = pa+\sum^{p}_{j=1}jb_j$. Applying the above arguments to all successive
$p$-rims of $\lambda$ we have that     
\begin{align*}\sum^k_{i=1}(h_i-2r_i+\epsilon_i) 
& \geq (p-2)a-(p-2)b_p - \sum^{p-1}_{j=1}(j-1)b_j\\
&\geq (2p-2)a + (2b_p + \sum^{p-1}_{j=1}b_j) - (pa+\sum^{p}_{j=1}jb_j) 
\\
&\geq (2p-2)a + (2b_p + \sum^{p-1}_{j=1}b_j) - n.
   \end{align*}
Observe that 
$$2b_p + \sum^{p-1}_{j=1}b_j \geq \frac{\sum^{p}_{j=1}jb_j}{p-1} 
  = \frac{n-pa}{p-1}.$$

Under the hypothesis, we can find an integer $t \geq (n+1)/(2p-1)$ such that 
$$\sum^{s}_{i=1}\lfloor \frac{\lambda_i - \lambda_{i+1}}{p} \rfloor \geq t.$$
Now observe that at least $t$ horizontal $p$-segments from the first $s$ rows
of $\lambda$ belong to these successive $p$-rims. Thus   $a \geq t \geq (n+1)/(2p-1)$. Hence,
$$\sum^k_{i=1}(h_i-2r_i+\epsilon_i) \geq (2p-2)a + \frac{n-pa}{p-1} -n \geq 
  \frac{p-2}{p-1} ((2p-1)a-n) > 0,$$
contradicting (\ref{hr}).
\end{proof}

\section{Structure of permutation modules} 
Throughout the section $n\geq 6$ is an {\it even} integer and $p=2$.  

We will study permutation modules $M_r$, mainly  for $r=1,2,3$. 
For $n\geq 2r$, let $\Specht_r\subset M_r$ denote the Specht module $S^{(n-r,r)}$
and (assuming $n>2r$)  
let $D_r=D^{(n-r,r)}$  be the unique simple quotient of $\Specht_r$.
Let $T_r\in M_r$ be the sum of all $r$-element sets.
Let 
$$\eta_{r,s}:M_r\to M_s$$ denote the incidence homomorphism
sending an $r$-set to the sum of $s$-sets incident with (i.e. containing or contained in) it. 
By \cite[Corollary 17.18]{JamesBook}, 
\begin{equation}\label{specht_module}
  \Specht_r=\cap_{t=0}^{r-1}\Ker\eta_{r,t}.
\end{equation}
We denote by $M'_r$ the {\em augmentation module}, i.e. the  submodule $\Ker\eta_{r,0}$ of $M_r$
(spanned by differences of pairs of basis elements).

The space $M_r$ has a natural bilinear form $\langle \cdot,\cdot\rangle_r$, with respect to which the standard basis is orthonormal. 
If we identify $M_r$ and $M_s$ with their dual spaces, using the corresponding bilinear forms, then $\eta_{s,r}$ is the dual map of $\eta_{r,s}$.
In particular, $\eta_{s,r}$ is injective iff  $\eta_{r,s}$ is surjective
and vice versa. Also $\image\eta_{s,r}\cong\image\eta_{r,s}^*$ as $\F \SSS_n$-modules.
We have 
\begin{equation}\label{adjoint}
\langle x, \eta_{r,s}(y) \rangle_s =\langle \eta_{s,r}(x), y\rangle_r \qquad (x \in M_s, y\in M_r).
\end{equation}
The ranks of the maps $\eta_{r,s}$ are given in \cite{Wilson}. We state the
special cases that we need of of this general result. 

\begin{lemma}\label{wilson} 
For $r\leq \min\{s, n-s\}$ we have
$$\rank_\F \eta_{r,s}= \sum_{1 \leq i \leq r,~\binom{s-i}{r-i}\mbox{ \tiny{is odd}}} 
                    \left(\binom{n}{i}-\binom{n}{i-1}\right).$$
In particular, $\eta_{1,3}$ is injective, and 
$$\rank_\F\eta_{1,2}=n-1=\dim \Specht_1,\quad 
  \rank_\F\eta_{2,3}=1+\frac{n(n-3)}{2}=1+\dim \Specht_2.$$
\end{lemma}

\begin{lemma} \label{M1} $M_1$ is a uniserial $\F \SSS_n$-module with socle layers 
$\F\cdot T_1 \cong \triv$, $D_1$, $\triv$. 
\end{lemma}

\begin{proof}This is well known and easy to show.
\end{proof}

Let $\F\SSS_n$-module $Q$ be defined by the short exact sequence 
\begin{equation}
\label{ses1}
0\to\F\cdot T_1\to M_1\to Q\to 0.
\end{equation}
In fact, $Q\cong \Specht_1^*$. 

The following lemma can be deduced from \cite[Theorem (1.1)]{MO}, but we give an independent proof for the reader's convenience:

\begin{lemma} \label{M2} As $\F \SSS_n$-modules, $\image \eta_{1,2} \cong Q$ and 
$M_2$ has the following structure:
\begin{enumerate}
\item[(i)] If $n \equiv 0\pmod 4$ then the composition factors of
$M_2$ are $\triv$ (twice), $D_1$ (twice) and $D_2$ (once).  Furthermore, $M_2/(\image\eta_{1,2}+\F\cdot T_2)\cong S_2^*$, whose socle is $D_2$.  

\item[(ii)] If $n \equiv 2\pmod 4$ then $M_2\cong \F \cdot T_2\oplus M_2'$,
and  $M_2'$ is uniserial with socle layers $D_1$, $\triv$, $D_2$, $\triv$, $D_1$.
\end{enumerate}
\end{lemma}

\begin{proof}
The dimension of $\image\eta_{1,2}$ is given by Lemma~\ref{wilson},
from which we see that it is isomorphic to $Q$. 
The composition factors are given by \cite{James}. 

(i) The second statement in (i) now follows using the facts that the dual Specht module $S^*_2$ is a quotient of $M_2$ and $S^*_2$ is uniserial with socle layers  $D_2$, $D_1$. 

(ii) The submodule $\F \cdot T_2$ is a direct summand of $M_2$ because $\binom{n}{2}$ is odd. Then  $M_2' =\Ker\eta_{2,0}$ is the complementary summand. 
The  composition factors of $M_2'$ are  $D_1$, $\triv$, $D_2$, $\triv$, $D_1$. Also we have
$\Specht_2\subseteq M_2'$. 
The composition factors of $\Specht_2$ are $D_2$, $D_1$, $\triv$ and $\Specht_2$
has a simple head isomorphic to $D_2$. Since ${M_2'}^{\SSS_n}=0$, it follows that
$\Specht_2$ must be uniserial with socle layers $D_1$, $\triv$, $D_2$.
The uniseriality of $M_2'$ and its socle layers now follow from the self-duality
of $M_2'$.  
\end{proof}

\begin{lemma}\label{M3}  The $\F \SSS_n$-module $M_3$ has the following structure.
\begin{enumerate} 
\item[(i)] If $n\equiv 0\pmod 4$, then $M_3\cong M_1\oplus U$, where $U$
is uniserial with socle layers $D_2$, $D_1$, $D_3$, $D_1$, $D_2$.

\item[(ii)] If $n\equiv 2\pmod 4$ then $\image \eta_{2,3}$ is uniserial
with socle layers $\triv$, $D_2$, $\triv$, $D_1$.
The composition factors of $M_3$ are $\triv$ (with multiplicity $4$), 
$D_1$ (twice), $D_2$ (twice) and $D_3$ (once). Also, $\soc(M_3)=\F\cdot T_3$.
\end{enumerate}
\end{lemma}
\begin{proof}
The composition factors are given by \cite{James}. The dimensions of
the images of the incidence maps $\eta_{r,s}$ are given by Lemma~\ref{wilson}.

\smallskip
(i) A simple computation shows that $\eta_{3,1}\circ\eta_{1,3}=1_{M_1}$,
so $\eta_{1,3}$ is a split injection of $M_1$ into $M_3$. 
In fact, by computing bilinear forms on the images of basis elements,
$\eta_{1,3}$ is seen to be an isometry. So
$M_3=\image\eta_{1,3}\oplus U$, with $U=(\image\eta_{1,3})^\perp=\Ker\eta_{3,1}$,
where the last equality is by (\ref{adjoint}).
  The module $U$ is a self-dual module and its  composition factors are
$D_1$ (twice), $D_2$ (twice) and $D_3$, and $U$ contains $\Specht_3$.
The structure of $U$ will follow from its self-duality if we prove
that $\Specht_3$ is uniserial with socle layers $D_2$, $D_1$, $D_3$. 
Since we know that the head of $\Specht_3$ is isomorphic to $D_3$, it suffices
to show that $D_1$ is not a submodule of $\Specht_3$. If it were, then we would have $\Hom_{\F\SSS_n}(D_1,M_3)\neq 0$, whence $\Hom_{\F\SSS_n}(M_3,D_1)\neq 0$, i.e. the fixed point subspace  $D_1^{\SSS_{n-3,3}}$ is non-trivial, which is easily checked to be false. 

\smallskip
(ii) We have $\image\eta_{1,2}\cong Q$. Also, it is easy to see that $\eta_{2,3}\circ\eta_{1,2}=0$, whence by dimensions using Lemma~\ref{wilson}, we conclude that $\image\eta_{1,2}=\Ker\eta_{2,3}$. Moreover, $\image\eta_{1,2}\cong Q$ by Lemma~\ref{M2}. 
Since $\image\eta_{1,2}\nsubseteq M_2'$, the structure of $Q$ implies that
$\Ker\eta_{2,3}\cap M_2'$ must be zero or isomorphic to $D_1$. Since we
know the rank of $\eta_{2,3}$, we see that the latter holds.
Thus from the structure of $M_2'$, we see that
$\eta_{2,3}(M_2')\cong M_2'/\soc(M_2')$. By dimensions,
we see that $\eta_{2,3}(M_2')=\image\eta_{2,3}$, so $\image\eta_{2,3}$
is uniserial with the socle layers as stated.

It remains to show that $\soc(M_3)=\F\cdot T_3$. For this, it suffices to
prove that $M_3$ has no submodule isomorphic to $D_1$, $D_2$ or $D_3$. 
For $D_1$ we explicitly check as in (i) that $D_1^{\SSS_{n-3,3}}=0$. 
The unique $D_3$ composition factor
of $M_3$ is the head of $\Specht_3$, so since $\Specht_3$ is not simple, it follows that
$M_3$ has no submodule isomorphic to $D_3$. 
Finally, we consider $D_2$. 
By the first part of (ii), $\image\eta_{2,3}$ has one composition factor $D_2$ 
as its second socle layer, so it suffices to show that $M_3/\image\eta_{2,3}$, which has a single $D_2$ composition factor, has no submodule isomorphic to $D_2$. However,   $M_3/\image\eta_{2,3}$ has $D_3$ as a composition factor, so maps surjectively onto $\Specht_3^*$.
By \cite{James},  $D_2$ is a composition factor of $\Specht_3^*$, but
$\Specht_3^*$ has a simple socle isomorphic to $D_3$. So $D_2$ is not 
a submodule of $M_3$. 
\end{proof}

Figures~\ref{ncong0} and \ref{ncong2} below are given for the reader's convenience, but they will not be used in proofs. The pictures give {\em partial}\, information on submodule structure of the permutation modules $M_2$ and $M_3$. The edges indicate the existence of uniserial subquotients. 

\begin{figure}[h]
$M_2:
\xymatrix@C=8pt@R=8pt{
&D_1\ar@{-}[dl]\ar@{-}[d]&\\
\triv&D_2\ar@{-}[d]&\triv\ar@{-}[dl]\\
&D_1&}
\qquad\qquad
M_3:
\xymatrix@C=8pt@R=8pt{
&&D_2\ar@{-}[d]\\
\triv\ar@{-}[d]&&D_1\ar@{-}[d]\\
D_1\ar@{-}[d]&\oplus&D_3\ar@{-}[d]\\
\triv&&D_1\ar@{-}[d]\\
&&D_2}$

\caption{Submodule structures for $n\equiv0\pmod4$\label{ncong0}}
\end{figure}

\bigskip

\begin{figure}[h]
$M_2:
\xymatrix@C=8pt@R=8pt{
&D_1\ar@{-}[d]\\
&\triv\ar@{-}[d]\\
\triv\ \oplus&D_2\ar@{-}[d]\\
&\triv\ar@{-}[d]\\
&D_1}
\qquad\qquad
M_3:
\xymatrix@R=8pt@C=8pt{
&&\triv\ar@{-}[dl]\ar@{-}[dr]&&\\
&D_1\ar@{-}[dl]&&D_2\ar@{-}[dl]\ar@{-}[dr]&\\
\triv\ar@{-}[dr]&&D_3\ar@{-}[dl]&&\triv\ar@{-}[dl]\\
&D_2\ar@{-}[dr]&&D_1\ar@{-}[dl]&\\
&&\triv&&}$
\caption{Submodule structures for $n\equiv2\pmod4$ \label{ncong2}}
\end{figure}

\begin{lemma}\label{Nker} 
We have:
\begin{enumerate}
\item[{\rm (i)}] $\image\eta_{1,3}$ is the unique submodule of
$M_3$ that is isomorphic to $M_1$ as $\F \SSS_n$-modules.
\item[{\rm (ii)}] $\Ker \eta_{3,1}$ is the unique submodule $N$ of 
$M_3$ such that $M_3/N\cong M_1$ as $\F \SSS_n$-modules.
\end{enumerate}
\end{lemma}

\begin{proof} Part (ii) follows from (i) by the duality of $\eta_{1,3}$ and $\eta_{3,1}$. For (i), note that 
$\dim\Hom_{\F \SSS_n}(M_1, M_3)=2$. The map $\eta_{1,3}$
and  the map $\beta$ sending each $1$-set to $T_3$ 
form a basis of this $\Hom$-space.  
Now $\eta_{1,3}$ is injective since $\soc(M_1)$ is spanned by $T_1$
and $\eta_{1,3}(T_1)=3T_3=T_3\neq 0$. Also we have 
$\image \beta\subset\image\eta_{1,3}$,
so $\image\eta_{1,3}$ is the unique submodule of $M_3$ isomorphic to $M_1$.
\end{proof}

In the following two lemmas, $N$ denotes the submodule of $M_3$ specified in Lemma~\ref{Nker}(ii).

\begin{lemma} \label{L0Mod4}
Let $n\equiv 0\pmod 4$. Then $\Ker \eta_{3,2}\cap N=\soc^3(N)$, and $$N/\soc^3(N)\cong \eta_{3,2}(N)= \image\eta_{3,2}\cap M'_2=\Specht_2.$$  
\end{lemma}
\begin{proof}
By Lemma~\ref{M3}, $N/\soc^3(N)$ is uniserial with socle $D_1$ and head $D_2$.
Since $M_2$ has no composition factor isomorphic to $D_3$, 
we have $\soc^3 (N)\subseteq\Ker\eta_{3,2}$.

We claim that the induced map $N/\soc^3(N)\to M_2$ is injective.
If not, its image is either zero or isomorphic to $D_2$. The latter is impossible
since $M_2$ has no submodule isomorphic to $D_2$. The former is also
impossible since it forces the rank of $\eta_{3,2}$ to be at most $\dim M_1 = n$,
contrary to Lemma \ref{wilson}, which gives the actual rank as $1+\dim \Specht_2$.

Thus the map $\eta_{3,2}$ induces an isomorphism of $N/\soc^3(N)$
with a submodule $\eta_{3,2}(N)\subseteq M_2$. 
Since $N\subseteq M_3'$ and 
$\eta_{2,0}\circ\eta_{3,2}=\eta_{3,0}$, we have $\eta_{3,2}(N)\subseteq M_2'$.
Comparing the dimensions, we see that $\eta_{3,2}(N)=\image\eta_{3,2}\cap M_2'$. 
This submodule of $M_2$ has the same dimension and the same composition factors 
as $\Specht_2$. 
Since any submodule of $M_2$ with $D_2$ as a composition  factor must contain $\Specht_2$, 
we now conclude that $\eta_{3,2}(N)=\Specht_2$.  
\end{proof}

\begin{lemma} \label{L2Mod4} 
Let $n\equiv 2\pmod 4$, $W=\image\eta_{2,3}\subseteq M_3$, and $Y := M'_2\subset M_2$. Then: 
\begin{enumerate}
\item[{\rm (i)}] $\soc^2(W)\subseteq N$ and $\soc^2(W)$ is uniserial  with socle layers $\triv, D_2$. 
\item[{\rm (ii)}] $\Ker(\eta_{2,3}|_Y) = \soc(Y)$, and $Y/\soc(Y)\cong \eta_{2,3}(Y) = W$.
\item[{\rm (iii)}] $N \cap W = N \cap \soc^3(W) = \soc^2(W),~~M_3/(N + \soc^3(W)) \cong Q.$
\item[{\rm (iv)}] We have $M_1\cong \image\eta_{1,3}\subseteq N$, $\image\eta_{1,3}\cap W=\F\cdot T_3$, and the submodule 
$$Q':=(\image\eta_{1,3}+\soc^2(W))/\soc^2(W)\subseteq N':=N/\soc^2(W)$$
is isomorphic to $Q$. 
Moreover,  $N'' :=N'/Q'$ is uniserial with socle layers $D_3,D_2$. 
\item[{\rm (v)}] We have $\soc^2(W)\subset \Specht_3\subset N$, the submodule  $D':=\Specht_3/\soc^2(W)\subset N'$ is isomorphic to $D_3$. Moreover,  $\soc(N') = D' \oplus \soc(Q') \cong D_3 \oplus D_1$ and $  N'/D' \cong N/\Specht_3 \cong \eta_{3,2}(N) \cong \Specht_2.$
\end{enumerate}

\end{lemma}
\begin{proof}
(i) The structure of $W$ is given in Lemma~\ref{M3}, which implies that $\soc^2(W)$ is uniserial  with socle layers $\triv, D_2$. 
Any nonzero quotient of $\soc^2(W)$ has $D_2$ as its head. 
But $M_3/N\cong M_1$ and $D_2$ is not a composition factor of $M_1$. So 
we see that $\soc^2(W)\subseteq N$.

(ii) Recall that $M_2 \cong Y \oplus \triv$ and $Y$ is uniserial with socle layers $D_1,\triv,D_2,\triv, D_1$  by Lemma 
\ref{M2}. Next, $\eta_{2,3}(Y)$ has codimension $\leq 1$ in $W = \image\eta_{2,3}$. 
Inspecting the submodule structures of $W$ and $Y$ given in Lemma \ref{M3}(ii)
and Lemma \ref{M2}, we see that $\eta_{2,3}(Y) = W$ and  
$\Ker(\eta_{2,3}|_Y) = \soc(Y)$. 

(iii) Note that $\eta_{3,1}\circ\eta_{2,3}\neq 0$, so 
$W \not\subseteq \Ker\eta_{3,1} = N$. Moreover, $\soc^3(W) \not\subseteq N$, since 
otherwise $W \cap N = \soc^3(W)$, and $M_3/N \cong M_1$ contains 
$(W+N)/N \cong W/\soc^3(W) \cong D_1$
as a submodule, which is a contradiction. As $W$ is uniserial and 
$\soc^2(W) \subseteq N$, it now follows that
$N \cap W = N \cap \soc^3(W) = \soc^2(W).$ 
Now, $M_3/(N+\soc^3(W))$ is a quotient of $M_3/N \cong M_1$ by 
$$(N+\soc^3(W))/N \cong \soc^3(W)/(N \cap \soc^3(W)) = \soc^3(W)/\soc^2(W) \cong
  \triv,
  $$
so this quotient must be isomorphic to $Q$.

(iv) We know that $N'$ has composition factors $D_3$, $D_2$, $D_1$, and $\triv$. It is easy to check that 
$\eta_{3,1}\circ\eta_{1,3}=0$, so $\image\eta_{1,3}\subseteq N = \Ker \eta_{3,1}$. By Lemma~\ref{Nker}(i), $\image\eta_{1,3}\cong M_1$. Using the submodule structure of $W$ and 
$\image\eta_{1,3}$ 
and the fact that $\soc(M_3)=\F\cdot T_3$, 
we conclude that $\image\eta_{1,3}\cap W=\F\cdot T_3$. 
Therefore the image $Q'$ of $\image\eta_{1,3}$ in
$N'$ is isomorphic to $Q$. 

Now we know that $N''$ has composition factors $D_3$ and $D_2$. 
Note  that $\Specht_3^*$ is a quotient of $M_3$, so some submodule $S'$ of 
$\Specht_3^*$ is a quotient of $N$. Also, $D_3$ is the head of $\Specht_3$ and the socle of 
$\Specht_3^*$. But $D_3$ is not a composition factor of $\soc^2(W)+\image\eta_{1,3}$. It 
follows that $\Specht_3^*$ is a quotient of $M_3/(\soc^2(W)+\image\eta_{1,3})$, 
whence $S'$ is a quotient of $N/(\soc^2(W)+\image\eta_{1,3})=N''$ and also of
$N'$. By \cite{James} the composition factors of $\Specht_3$ (and 
$\Specht_3^*$) are $D_3$, $D_2$, $\triv$. Among these, only $\triv$ is a composition factor of 
$M_1 \cong M_3/N$, so $S'$ has both $D_3$ and $D_2$ as composition factors. 
We have shown that $S'$ is a quotient of $N''$ which has exactly two composition
factors $D_3$ and $D_2$. It follows that in fact $N''\cong S'$. In this case, 
$\soc(N'') \cong D_3$ since $\soc(\Specht_3^*)\cong D_3$ is simple. 

(v) Recall that $\Specht_3 \subset M_3$ and ${\mathrm {head}}(\Specht_3) \cong D_3$ is not
a composition factor of $M_1 \cong M_3/N$. It follows that  
$\Specht_3\subset N$ and, since $D_1$ is not a composition factor of $\Specht_3$,
the image $D'$ of $\Specht_3$ in $N'$ intersects $Q'$ trivially. The aforementioned 
structure of $N''$ implies that $\Specht_3$ has no quotient isomorphic to $N''$.
Therefore, under the natural projection $N' \to N''$, $D'$ projects onto 
a module isomorphic to $D_3$, or $0$. The latter cannot happen since $D_3$ is not a composition factor of  $\soc^2(W) + \image \eta_{1,3}$. So $D'$ projects onto a module
isomorphic to $D_3$. This implies that the composition factors of $D'+Q'$ are 
$D_3$, $D_1$ ,$\triv$. Since $D'\cap Q'=0$ and $Q'\cong Q$, it follows that $D'\cong D_3$.
We have shows that that $D_3\oplus D_1\cong D'\oplus \soc(Q')\subseteq\soc(N')$. 
On the other hand, $N'/Q'=N''$, $\soc(N'')\cong D_3$, and $\soc (Q')\cong D_1$. Together these imply that $\soc(N')$ embeds into $D_3\oplus D_1$. So $D'\oplus \soc(Q')=\soc(N')$.

Let $\pi$ denote the natural projection $N \to N'$. Then we have shown 
that $\Ker(\pi|_{\Specht_3})$ has two composition factors $D_2$ and $\triv$. On the other
hand, $\Ker\pi = \soc^2(W)$. It follows that $\Ker(\pi|_{\Specht_3}) = \soc^2(W)$
so that in fact $D'=\Specht_3/\soc^2(W)$ and $N'/D' \cong N/\Specht_3$.
By (\ref{specht_module}) we have $\Specht_3=N\cap\Ker\eta_{3,2}$. So 
$N/\Specht_3= N/(\Ker\eta_{3,2}\cap N) \cong \eta_{3,2}(N)$
is a submodule of $M_2$ with composition factors  $D_1$, $\triv$, $D_2$, which are 
precisely the composition factors of $\Specht_2$. Hence,
$N'/D' \cong N/\Specht_3 = \eta_{3,2}(N)=\Specht_2$.
\end{proof}

\section{Main reduction theorem}

The following theorem is the main tool for proving reducibility of various restrictions $\Res^{\SSS_n}_{X}$ in the key case $p=2|n$. Note by Theorem~\ref{hom} that the assumption $d_3(V)>d_1(V)$ is equivalent to the assumption that $V$ is not trivial and not basic spin module. 

\begin{theorem}\label{prereduction}
Let $p=2|n \geq 6$, $V$ be an irreducible $\F \SSS_n$-module satisfying $d_3(V)>d_1(V)$, and $X$ be a subgroup of $\SSS_n$. Let $N$ be the $\F \SSS_n$-submodule of $M_3$ specified in Lemma \ref{Nker}(ii).   
Suppose that for any nonzero $\F\SSS_n$-quotient $J$ of $N$ we have 
$J^X \neq 0$ and if $J^{\SSS_n} \neq 0$ then $\dim J^X \geq 2$.
Then the restriction $\Res^{\SSS_n}_X V$ is reducible.
\end{theorem}

\begin{proof}
Set $E := \End_{\F}(V)$ so that $d_r(V)=\dim\Hom_{\F\SSS_n}(M_r,E)$. By Schur's Lemma, $\dim E^{\SSS_n} = \End_{\F\SSS_n}(V)=1$, so 
the $\F\SSS_n$-module $E$ contains a unique submodule $E_1 \cong \triv$. Note that $E^X=\End_{\F X}(\Res^{\SSS_n}_X V)$, so it suffices to prove that $\dim E^X\geq 2$.

By definition of $N$ in Lemma \ref{Nker}(ii), we have an exact sequence 
$$
0\to N\to M_3\to M_1\to 0.$$ Applying $\Hom_{\F\SSS_n}(-,E)$ to this sequence and using the assumption $d_3(V)>d_1(V)$, we conclude that there is some $f \in \Hom_{\F\SSS_n}(M_3,E)$ such 
that $J := f(N) \neq 0$. 

If  $J \cap E_1 = 0$,
then $E$ contains a submodule isomorphic to $J \oplus E_1$, whence 
$\dim E^X \geq 2$ as $J^X\neq 0$ by assumption. On the other hand, if $J \cap E_1 \neq 0$ then $J^{\SSS_n} \neq 0$. In this case
$\dim E^X \geq \dim J^X \geq 2$  
by assumption again.
\end{proof}

Our main goal now will be to obtain permutation group theoretic conditions on the subgroup $X$ which guarantee that the assumptions of Theorem~\ref{prereduction} hold. 

To bound dimensions of various fixed point subspaces, we will 
frequently use the following well-known estimates:

\begin{lemma}\label{inv-easy}
Let $X$ be a group and $U \supseteq V$ be $\F X$-modules. Then: 
\begin{enumerate}
\item[{\rm (i)}] $\dim (U/V)^X - \dim H^1(X,V) \leq \dim U^X - \dim V^X \leq \dim (U/V)^X.$
\item[{\rm (ii)}] If in addition $X$ acts trivially on $V$ and $\Hom(X,\F)=0$, then 
$$\dim U^X = \dim V + \dim (U/V)^X.$$
\end{enumerate}
\end{lemma}

\begin{proof}
(i) follows from the exact sequence $0\rightarrow V^X\rightarrow U^X\rightarrow (U/V)^X\rightarrow H^1(X,V)$. 

(ii) In this case $H^1(X,V) =\Hom(X,\F)= 0$, whence the statement follows from (i). 
\end{proof}

Let $X\leq \SSS_n$ be any subgroup and $1\leq r\leq n/2$. We set 
\begin{equation}\label{EFE}
f_r(X):=\dim (M_r)^X, \quad f_0(X) := 0,\quad e_r(X):=f_r(X)-f_{r-1}(X).
\end{equation}
Note that $f_r(X)$ is the number of $X$-orbits on $\Omega_r$. Recalling the module $Q$ defined in (\ref{ses1}), when $p=2|n$ we also put 
\begin{equation}\label{EH}
h(X): = \dim H^1(X,Q).
\end{equation}

For any partition $\lambda \vdash n$, 
let $\chi^{\lambda}$ denote the irreducible ordinary $\SSS_n$-character 
labeled by $\lambda$. It is well known that the $\SSS_n$-character
afforded by the permutation module $\C\Omega_r$ is 
$\sum^{r}_{s=0}\chi^{(n-s,s)}$. Denote $\alpha := \chi^{(n-1,1)}$, so that 
$\alpha+1_{\SSS_n}$ is the permutation character of $\SSS_n$ acting on 
$\Omega := \{1,2, \ldots ,n\}$. Applying \cite[Lemma 3.3]{GT2}, we see that
$$\sum^{r}_{s=0}\chi^{(n-s,s)} = \left\{ 
   \begin{array}{ll}S^2(\alpha) & \text{if $r = 2$},\\
                    S^3(\alpha) - \wedge^2(\alpha)-\alpha & \text{if $r = 3$},
   \end{array} \right.$$
where $S^k$ denotes the $k$th symmetric power and $\wedge^k$ denotes the $k$th 
exterior power.   
If we know the restriction $\alpha_X := \Res^{\SSS_n}_X \alpha$ explicitly, we can 
compute $f_2(X)$ and $f_3(X)$ by computing the scalar product of $X$-characters as follows:
\begin{equation}\label{fixed}
  f_2(X) = [S^2(\alpha_X),1_X]_X,~~f_3(X) = 
    [(S^3(\alpha_X) - \wedge^2(\alpha_X)-\alpha_X,1_X]_X.
\end{equation}

Next we record some obvious observations:


\begin{lemma}\label{e22}
Let $X \leq \SSS_n$ be a transitive subgroup. Then: 
\begin{enumerate}
\item[{\rm (i)}] $f_2(X) = 1$ if and only if $X$ is $2$-homogeneous.
\item[{\rm (ii)}] Suppose $|X|$ is even. Then $f_2(X) = 1$ if and only if $X$ is $2$-transitive.
\item[{\rm (iii)}] $f_2(X) \leq 2$ if $X$ is a rank $\leq 3$ subgroup of
$\SSS_n$.  
\end{enumerate}
\end{lemma}

\begin{proof}
(i) is obvious: $X$ is $2$-homogeneous precisely when it acts transitively on 
$\Omega_2$. For (ii), observe that $X$ contains an involution $t$, and so 
we can find $x,y \in \Omega$ interchanged by $t$. It follows that $X$ is $2$-homogeneous 
precisely when it is $2$-transitive.

For (iii), note that if $X$ is a rank $2$ subgroup $f_2(X) = 1$ by (ii). If $X$ is a rank $3$ subgroup, 
then the point stabilizer of 
$x \in \Omega$ in $X$ has two orbits on $\Omega \setminus \{x\}$, whence
$X$ has at most two orbits on $\Omega_2$.
\end{proof}

Note that since $p=2$ in the following proposition, the condition $X=O^2(X)$ is equivalent to the condition $\Hom(X,\F)=0$ from Lemma~\ref{inv-easy}. In many applications $X$ will be perfect, in which case this assumption of course holds. 

\begin{proposition}\label{inv}
Let $p=2{|}n\geq 6$, $N$ the $\F \SSS_n$-submodule of $M_3$  
specified in Lemma \ref{Nker}(ii), and let $J$ be any nonzero $\F\SSS_n$-quotient of $N$. 
Suppose that $X=O^2(X) \leq \SSS_n$ is a 
subgroup such that 
$$f_1(X) = 1,~~e_3(X) \geq h(X)+1,\mbox{ and either }f_2(X) \geq 3\mbox{ or }\Specht_2^X \neq 0.$$  
Then $J^X \neq 0$. Moreover, if $J^{\SSS_n} \neq 0$, then $\dim J^X \geq 2$.
\end{proposition}
\begin{proof}
We write $f_r$ for $f_r(X)$, $e_r$ for $e_r(X)$, and $h$ for $h(X)$. 
Note that $Q^X=0$ and $D_1^X = 0$ since $f_1 = 1$ and 
$X = O^2(X)$. Combining this with the structure of $M_2$ given in Lemma \ref{M2}
and applying Lemma \ref{inv-easy}, we see that 
\begin{equation}\label{d2}
  f_2 \geq 1+\dim D_2^X-\dim H^1(X,Q) = 1+ \dim D_2^X-h.
\end{equation}

Note that $D_2^X\neq 0$. Indeed, if $f_2\geq 3$, this follows by considering composition factors of $M_2$ described in Lemma \ref{M2} and using $Q^X=D_1^X=0$. On the other hand, if $\Specht_2^X\neq 0$, this follows by considering composition factors of $\Specht_2$ using $D_1^X=0$. 

\medskip
{\bf Case 1:} $n\equiv 0\pmod 4$. Then  $N$ 
is uniserial by Lemma~\ref{M3}, and we are going to check that $J^X\neq 0$ for each of its five non-trivial quotients $J$. This is all we have to do, since $\triv$ is not a composition factor of $N$, and so we never have $J^{\SSS_n} \neq 0$. 

Note that $\soc(N)\cong D_2$, so $N^X\supseteq D_2^X\neq 0$.
By assumption, we have $f_3 \geq f_2+h+1$, so (\ref{d2}) implies  
$f_3\geq \dim D_2^X+2$. Since $M_3 = M_1 \oplus N$ and 
$f_1 = 1$, it follows that $(N/\soc(N))^X\neq 0$.
Then since $D_1^X=0$, we also get $(N/\soc^2(N))^X\neq 0$.
Next, $N/\soc^3(N)\cong \Specht_2$ by Lemma~\ref{L0Mod4}. 
If $\Specht_2^X\neq 0$, we are done. Otherwise, the conditions $f_2\geq 3$ and $D_1^X = 0$ imply that 
$(N/\soc^3(N))^X\neq 0$. Finally, $N/\rad(N)\cong D_2$ and we 
already have $D_2^X\neq 0$. 

\medskip
{\bf Case 2:} $n\equiv 0\pmod 4$. We are going to use the notation of Lemma~\ref{L2Mod4}. 

\smallskip
{\em Step 1:} we prove that 
$J^X \neq 0$ for any nonzero quotient $J=N/K$ of $N$.

By Lemma~\ref{L2Mod4}(i), we have the submodule $\soc^2(W)\subset N$ which is uniserial with socle layers $\triv, D_2$. So any nonzero quotient of $\soc^2(W)$ either 
contains $\triv$ or is isomorphic to $D_2$, hence it contains nonzero 
$X$-fixed points. In particular, $(N/K)^X > 0$ if $\soc^2(W) \not\subseteq K$, and we may now assume that $\soc^2(W)\subseteq K$. In other words, we are reduced to showing 
that $X$ has nonzero fixed points on every nonzero $\F \SSS_n$-quotient of 
$N'=N/\soc^2(W)$. 

Recall that $M_2 \cong Y \oplus \triv$, see Lemma~\ref{L2Mod4}. In particular, $\dim Y^X=f_2-1$, and so 
$\dim (\soc^4(Y))^X=f_2-1$, since $X$ has no fixed points on $U/\soc^4(U)\cong D_1$.
Applying Lemma \ref{inv-easy}(i) to the exact sequence 
$$0 \to \soc^2(Y) \to \soc^4(Y) \to \soc^4(Y)/\soc^2(Y)\to 0$$
with $(\soc^2(Y))^X \cong Q^X = 0$, we see that
\begin{equation}\label{le3}
  \dim (\soc^4(Y)/\soc^2(Y))^X\leq f_2+h-1.
\end{equation}
By Lemma~\ref{L2Mod4}(ii), we have $\eta_{2,3}(Y) = W\cong Y/\soc(Y)$. So  
$\soc^3(W) \cong \soc^4(Y)/\soc(Y)$ is an extension of  
$\soc^4(Y)/\soc^2(Y)$ by $\soc^2(Y)/\soc(Y)\cong \triv$. Together with (\ref{le3}) and Lemma \ref{inv-easy},
this implies that 
\begin{equation}\label{s3wx}
  \dim (\soc^3(W))^X \leq f_2+h. 
\end{equation} 
Since $Q^X = 0$, Lemma~\ref{L2Mod4}(iii) yields 
\begin{equation}\label{ns3w}
  \dim(N + \soc^3(W))^X = f_3.
\end{equation} 
Moreover, by the same lemma, we have
$$N' = N/\soc^2(W) = N/(N \cap \soc^3(W)) \cong (N+\soc^3(W))/\soc^3(W).$$
Since $f_3-f_2 = e_3 \geq h+1$ we deduce from (\ref{s3wx}) and (\ref{ns3w})  that $(N')^X \neq 0$.

Now we apply Lemma~\ref{L2Mod4}(iv).
Since $N'' = N'/Q'$ and $Q'^X \cong Q^X =0$, we have that $(N')^X \neq 0$ implies $N''^X\neq 0$. 
Recalling that $D_2^X\neq 0$ and ${\mathrm {head}}(N'') \cong D_2$, we have now 
shown that $X$ has nonzero fixed points on every nonzero quotient of $N''$.

It remains to consider quotients of $N'$ by nonzero submodules $R'$ which do not 
contain $Q'$. Since $Q'$ has a simple socle and $Q'/\soc(Q')\cong\triv$, we only need to  
consider $R'$ such that $R'\cap Q'=0$. We have shown in 
Lemma~\ref{L2Mod4}(v) that $\soc(N') = D' \oplus \soc(Q')$. Hence the condition $R' \cap Q'=0$ 
implies that $R' \supseteq D'$. So we must show that $X$ has nonzero fixed points 
on every nonzero quotient of the $\F \SSS_n$-module $N'/D' \cong \Specht_2$, see Lemma~\ref{L2Mod4}(v) again. By assumption, $\Specht_2^X\neq 0$ unless $f_2\geq 3$, in which case $D_1^X=0$ implies $\Specht_2^X\neq 0$.
The only proper quotients of $\Specht_2$ have socles $\triv$ and $D_2$ and we have seen already that $X$ has non-trivial fixed points on these simple modules. 

\smallskip
{\em Step 2:} we assume that $J=N/K$ contains  
a trivial $\SSS_n$-submodule $E \cong \triv$ and prove that $\dim J^X \geq 2$. We again use the notation of Lemma~\ref{L2Mod4}. 
Since we have shown that $D_2^X, (N'')^X \neq 0$, by 
Lemma \ref{inv-easy}(ii) it suffices to show that $J$ contains an
$\F\SSS_n$-submodule $L$ where $L/E$ is isomorphic to $D_2$ or $N''$.  This is obvious if $K = 0$. So we may assume
that $K \supseteq \soc(W) = \soc(N) \cong \triv$. 

If $K \cap \soc^2(W) = \soc(W)$, then 
$J$ contains $\soc^2(W)/\soc(W)\cong D_2$, and we can take 
$L \cong D_2 \oplus E$. 
Thus we may assume that $K\supseteq \soc^2(W)$ and so $J$ 
can be regarded as a quotient $N'/K'$ of $N'=N/\soc^2(W)$.
Let $P'$ be the preimage of $E$ in $N'$ so that $E \cong P'/K'$. 
Since $\triv$ is not a 
composition factor of $N'/Q'$, we see that $Q'+K' \supseteq P'$, whence
$Q'/(Q' \cap K') \cong (Q'+K')/K'$ contains $E = P'/K'$. But $Q'$ has 
socle layers $D_1,\triv$, so $Q' \cap K' = \soc(Q')$
and $P' = Q'+K'$. Recall that $N'/Q' = N''$ is uniserial with 
socle layers $D_3$, $D_2$. 
So if $P'\neq N'$ then $J\cong N'/K'$ is isomorphic to an extension of the nonzero quotient $N'/P'$ of $N'/Q'\cong N''$  by $P'/K'=E$,  and  we may take $L=J$.

Assume $P' = N'$, so that $N'/K' \cong \triv$. Since $D' \cong D_3$ is simple, 
it follows that $D' \subset K'$. Now we see that $N'/K' \cong \triv$ is a
quotient of $N'/D'$ which is isomorphic to $\Specht_2$ by 
Lemma~\ref{L2Mod4}(v), a contradiction since ${\mathrm {head}}(\Specht_2) \cong D_2$. 
\end{proof}

Now we can prove the main result of this section:

\begin{theorem}\label{reduction}
Let $p=2|n \geq 6$, $V$ be an irreducible $\F\SSS_n$-module 
satisfying $d_3(V)>d_1(V)$, and 
$X=O^2(X) \leq \SSS_n$ be a subgroup such that 
$$f_1(X) = 1,~~e_3(X) \geq h(X)+1,\mbox{ and either }f_2(X) \geq 3\mbox{ or }\Specht_2^X \neq 0.$$  
Then the restriction $\Res^{\SSS_n}_X V$ is reducible.
\end{theorem}

\begin{proof}
Apply Theorem~\ref{prereduction} and Proposition \ref{inv}. \end{proof}

Next we show that the condition $\Specht_2^X \neq 0$ is a fairly mild condition which 
{\it always} holds for rank $3$ permutation groups of even degree $n$:

\begin{lemma}\label{rank3}
Let $2|n \geq 6$ and let $G \leq \SSS_n$ be a rank $3$ permutation group.
Then $\Specht_2^G \neq 0$.
\end{lemma}

\begin{proof}
Since $2|n$, $|G|$ is even and so $G$ contains an involution $j$. Choose
$x_0 \in \Omega = \{1,2, \ldots,n\}$ {\it not} fixed by $j$.
By assumption, the stabilizer $G_{x_0}$ of $x_0$ 
has two orbits $\Delta_1$, $\Delta_2$ on $\Omega \setminus \{x_0\}$. Since
$|\Delta_1| + |\Delta_2| = n-1$ is odd, we may assume that $|\Delta_1|$ 
is odd and $|\Delta_2| $ is even. Since $|G|$ is even, it follows 
by Lemma \ref{e22}(ii), (iii) that $G$ has exactly two orbits 
$\Phi$ and $\Psi$ on the set $\Omega_2$ of $2$-subsets $\{x,y\}$ of $\Omega$, where
$\{x_0,z\} \in \Psi$ if and only $z \in \Delta_2$. Note that
$$|\Phi| = n|\Delta_1|/2,~~|\Psi| = n|\Delta_2|/2.$$
(Indeed, suppose for instance that $j(x_0) =: z \in \Delta_1$. By counting we see that the number of 
{\it ordered} pairs $(x,y)$ with $\{x,y\} \in \Phi$ is $n|\Delta_1|$. Since $j$ interchanges 
$x_0$ and $z$, it follows that $|\Phi| = n|\Delta_1|/2$ and so $|\Psi| = n|\Delta_2|/2$. The same
argument applies if $z \in \Delta_2$.)

Now $\Phi \cup \Psi$ forms a basis for $M_2$. Since $\Psi$ is a $G$-orbit, it suffices to show that the orbit sum $\hat\Psi := \sum_{\{x,y\} \in \Psi}\{x,y\}$
belongs to $\Specht_2$.  A standard fact about Specht modules following from 
(\ref{specht_module}) and (\ref{adjoint}) is that
$$\Specht_2 = \langle T_2 \rangle^{\perp} \cap \eta_{1,2}(M_1)^\perp,$$
where perpendicularity is with respect to the natural inner product $\langle \cdot,\cdot\rangle_2$ on
$M_2$. Now,
$$\langle \hat\Psi,T_2 \rangle_2 = |\Psi| = n|\Delta_2|/2 \equiv 0 (\bmod 2).$$
Next, 
$$\langle \hat\Psi,\eta_{1,2}(x_0)\rangle_2 = |\Delta_2| \equiv 0 (\bmod 2),$$
and similarly $\langle  \hat\Psi,\eta_{1,2}(x)\rangle_2 = 0$ for all $x \in \Omega$.
\end{proof}

\section{Special embeddings of $\AAA_m$ into $\AAA_n$}
Let $X\cong \AAA_m$ with $m \geq 5$. 
In this section, we consider two {\em special}\, kinds of embeddings of $X$ into symmetric groups $\SSS_n$. The first arises from the action of $X$ on $k$-subsets of 
$$\Delta := \{1,2, \ldots, m\}$$ 
with $2 \leq k < m/2$, giving rise to an embedding
of $X$ into $\SSS_n$, where $n=\binom{m}{k}$.  
The second embedding comes from the action on set partitions of $\Delta$ into
$b$ subsets of size $a$, where $m=ab$. This gives rise to an embedding into
$\SSS_n$, where $n=(ab)!/((a!)^b \cdot b!)$. 

\begin{lemma}\label{start}
Suppose that $m \geq 8$. Then for any of the two special embeddings, we have 
that $n \geq m(m-1)/2$.
\end{lemma}

\begin{proof}
For the first embedding, observe that the sequence $\binom{m}{k}$ is
increasing for $2 \leq k < m/2$, whence $\binom{m}{k} \geq \binom{m}{2}$. 
For the second embedding, denote $N_{a,b}:= (ab)!/((a!)^b \cdot b!)$ and 
observe that 
\begin{equation}\label{nab}
  \frac{N_{a,b+1}}{N_{a,b}} = \frac{(ab+1)(ab+2) \ldots (ab+a-1)}{(a-1)!}
    \geq ab+1 \geq 5
\end{equation}
as long as $a,b \geq 2$. An induction on $a$ shows that 
$N_{a,2} > \binom{2a}{2}$ for $a \geq 4$. Now another induction on $b \geq 2$ 
using (\ref{nab}) shows that $N_{a,b} > \binom{ab}{2}$ whenever $a \geq 4$
and $b \geq 2$. Next, an induction on $a$ shows that 
$N_{a,3} \geq \binom{3a}{2}$ for $a \geq 2$, with equality only when $a = 2$. 
Now another induction on $b \geq 3$ using (\ref{nab}) shows that 
$N_{a,b} > \binom{ab}{2}$ if $a \geq 2$, $b \geq 3$, and $ab \geq 8$.
\end{proof}

Recall that the $\SSS_n$-module $Q$ defined in (\ref{ses1}) and the integer $h(X)$ defined in (\ref{EH}). 

\begin{lemma}\label{h1}
Let $p=2$. We have:
\begin{enumerate}
\item[{\rm (i)}] $\dim H^2(X,\F)=1$.
\item[{\rm (ii)}] If $n$ is even, then one of the following statements holds:
\begin{enumerate} 
\item[(a)] $\dim H^1(X,M_1)=1$ and $h(X) \leq 2$.
\item[(b)] $4|m$, $\AAA_m$ embeds into $\AAA_n$ via its action on partitions 
$(m/2,m/2)$ of $\Delta$, and 
$\dim H^1(X,M_1) = 2$, $h(X) \leq 3$.
\end{enumerate}
\end{enumerate}
\end{lemma}

\begin{proof}
(i) is a well-known fact about the Schur multiplier of $\AAA_m$.

(ii) By Frobenius reciprocity we have that 
$$H^1(X,M_1) \cong H^1(X_1,\F) \cong \Hom(X_1,(\F,+)),$$ 
where $X_1$ is the stabilizer in $X$ of a point on the set $\Omega$,
and $(\F,+)$ is the additive group of the field $\F$.
First we consider the case where $X$ is acting on $k$-sets of 
$\Delta$. Then 
$$X_1 = (\SSS_k\times \SSS_{m-k}) \cap \AAA_m \cong 
        (\AAA_k \times \AAA_{m-k}) \cdot 2.$$
Since $p= 2$, we have that $\Hom(\AAA_s,(\F,+)) = 0$ for all $s \geq 1$. Denoting by $C_2$ the group of order $2$, it follows
that 
$\Hom(X_1,(\F,+)) \cong \Hom(C_2,(\F,+))$ 
is one-dimensional.

Next we consider the case $X$ is acting on set partitions of 
$\Delta$ into $b \geq 2$ subsets of size $a = m/b \geq 2$. Then
$$X_1 = (\SSS_a \wr \SSS_b) \cap \AAA_m.$$
We may assume that the transposition $(1,2)$ fixes the set partition 
fixed by $X_1$. Then $(1,2)$ belongs to the base subgroup $B = \SSS_a^b$, whence
$[B : B \cap X_1] = 2$ and $X_1 \cong (B \cap X_1) \cdot \SSS_b$.
As mentioned above, $\Hom(\AAA_a,(\F,+)) = 0$. Hence
$$\Hom(X_1,(\F,+)) \cong \Hom(X_1/\AAA_a^b,(\F,+)) = \Hom(Y,(\F,+)),$$
where $Y = X_1/\AAA_a^b \cong 2^{b-1}\cdot \SSS_b$. If $b \geq 3$, then one
can check that $[Y,Y]$ contains the normal subgroup $2^{b-1}$, and so
$$\Hom(Y,(\F,+)) \cong \Hom(\SSS_b,(\F,+)) \cong 
  \Hom(C_2,(\F,+)) \cong \F$$
as $\Hom(\AAA_b,(\F,+)) = 0$. Assume that $b = 2$, i.e. $X_1$ fixes the 
partition $\Delta = \{1,2, \ldots ,a\} \cup \{a+1, \ldots ,m\}$. If 
$a \geq 3$ is odd, then the permutation 
$$g~:~(1,a+1,2,a+2)(3,a+3)\ldots (a,2a)$$
belongs to $X_1$ and $g^2 = (1,2)(a+1,a+2)$ gives rise to an involution 
in $Y$. Thus $Y$ is cyclic of order $4$, and so again 
$\Hom(Y,(\F,+)) \cong \F$. If $a$ is even, then $Y$ (of order $4$) is
generated by two involutions $(1,2)(a+1,a+2)$ and 
$(1,a+1)(2,a+2) \ldots (a,2a)$, whence
$\Hom(Y,(\F,+)) \cong \F^2$. This proves the claims on $\dim H^1(X,M_1)$ in (ii). 

Now the bounds on $h(X)$ in (ii) follow immediately from the portion
$$0=H^1(X,\F)\to H^1(X,M_1)\to H^1(X,Q)\to H^2(X,\F)$$
of the long exact sequence arising from (\ref{ses1}).
\end{proof}

\begin{lemma}\label{LFSubsets}
Let $X$ embed into $\SSS_n$ via its action on $k$-subsets of $\Delta$ for $2 \leq k < m/2$.
Then $f_2(X) = k$. 
\end{lemma}

\begin{proof}
We claim that the orbits of $X$ or of $\SSS_m$ on pairs $\{A,B\}$ of 
distinct $k$-subsets of $\Delta$ are labeled by $j := |A \cap B|$ for $0 \leq j \leq k-1$, hence $f_2(X) = k$. Indeed, the claim is obvious for $\SSS_m$. 
Since $A \neq B$, we can find $i \in A \setminus B$
and $j \in B \setminus A$. Now the transposition 
$(i,j)$ fixes the pair $\{A,B\}$, and so 
$\SSS_m$ and $\AAA_m$ have the same orbits on pairs $\{A,B\}$.
\end{proof}

Next we handle the embedding of $\AAA_m$ into $\AAA_n$ via its action on $2$-subsets:
 
\begin{corollary}\label{pairs2}
Let $p=2$ and $m \geq 6$ be such that $n:={m\choose 2}$ is even. Let $X = \AAA_m$ embed into $\AAA_n$ via its
action on $2$-subsets of $\{1,2, \ldots ,m\}$. Suppose that an irreducible  
$\F\SSS_n$-module $V$ satisfies the condition $d_3(V)>d_1(V)$.  
Then $\Res^{\SSS_n}_X V$ is reducible.
\end{corollary}

\begin{proof}
By Lemma \ref{h1}(ii), we have that $h(X) \leq 2$. On the other hand, $f_1(X) = 1$, and 
$f_2(X) = 2$ by Lemma \ref{LFSubsets}. Also,  $f_3(X) \geq 5$. Indeed, we can regard the 
$\F \SSS_n$-permutation module $M_1$ as having a basis consisting of all $2$-subsets $\{i,j\}$ of 
$\Delta = \{1,2, \ldots ,m\}$. Then the module $M_3$ has a basis consisting of unordered triples of distinct pairs, and $\SSS_m$ has $5$ orbits on this set represented by the triples
$$
\begin{aligned}
&\{\{1,2\}, \{3,4\}, \{5,6\}\},\qquad  \{\{1,2\}, \{3,4\}, \{4,5\}\}
,\qquad  \{\{1,2\}, \{2,4\}, \{3,4\}\}\\
&\{\{1,2\}, \{1,3\}, \{1,4\}\}
,\qquad  \{\{1,2\}, \{1,3\}, \{2,3\}\}.
\end{aligned}
$$
In particular, $e_3(X) \geq 3 \geq h(X)+1$; furthermore, $\Specht_2^X \neq 0$ by Lemma \ref{rank3}. Hence we are done by Theorem \ref{reduction}. 
\end{proof}

\begin{lemma}\label{e2}
Suppose that $m \geq 11$. Then $e_2(X)\geq 2$, unless
$\AAA_m$ embeds in $\AAA_n$ via its action on $2$-subsets of 
$\Delta$, in which case $e_2(X) = 1$.
\end{lemma}

\begin{proof}
Recall that $f_1(X) = 1$ for the special embeddings of $X$ into $\SSS_n$ in question.
Now for the action of $X$ on $k$-subsets of $\Delta$ the result follows from Lemma~\ref{LFSubsets}.

Now let $X$ act on partitions $P = \{P_1, \ldots ,P_b\}$ of $\Delta$ into
$b$ $a$-subsets $P_1, \ldots ,P_b$. We will exhibit at least $3$ orbits of
$X$ on pairs of partitions $\{P,Q\}$. Note that $\Delta$ admits 
two partitions with no common subset between them. It follows that, for
each $j = 0,1, \ldots,b-2$, $\Delta$ admits a pair of partitions 
$\{P,Q\}$, where $P$ and $Q$ contain exactly $j$ common subsets. Certainly,
such pairs with different parameters $j$ belong to different $\SSS_m$-orbits.
In particular, we are done if $b \geq 4$.

Suppose $b = 3$ and $a \geq 4$. Then we get at least one orbit with 
the above parameter $j = 0$. For $j = 1$, we get at least two orbits 
with representatives $\{P,Q\}$, where 
$P = \{P_1,P_2,R\}$, $Q = \{Q_1,Q_2,R\}$, and $|P_1 \cap Q_1| = 1$,
respectively $|P_1 \cap Q_1| = 2$. 

Suppose $b = 2$ and $a \geq 6$. Then for each $s = 1,2,3$ we get at least one
orbit with representatives $\{P,Q\}$, where 
$P = \{P_1,P_2\}$, $Q = \{Q_1,Q_2\}$, and $|P_1 \cap Q_1| = s$. 
\end{proof}

\begin{remark} 
{\rm 
It is easy to check that for the embedding via the action of $X$ on $b$ $a$-subsets, we have $e_2(X) = 0$ if $(a,b) = (3,2)$, and $e_2 = 1$ if 
$(a,b) = (2,3)$, $(4,2)$, or $(5,2)$. On the other hand, 
$e_2(X) = 3$ if $(a,b) = (3,3)$, as one can compute using (\ref{fixed}) below and 
\cite{GAP}.    
}
\end{remark}

\begin{lemma}\label{e31}
Suppose $m \geq 6$ and $X = \AAA_m$ embeds into $\SSS_n$ via its action on
$k$-subsets of $\Delta$ for $2 \leq k < m/2$. Then either $e_3(X) \geq 4$,
or $k = 2$ and $e_3(X) = 3$.  
\end{lemma}

\begin{proof}
By Lemma \ref{LFSubsets}, we have  $f_2(X) = k$.
So we will try to exhibit at least $(k+4)$ $\SSS_m$-orbits on triples 
$\{A,B,C\}$ of $k$-subsets. We may assume that 
$$|A \cap B| \geq |A \cap C| \geq |B \cap C|$$ 
and call $(|A \cap B|,|A \cap C|,|B \cap C|)$ the {\it mark} of the triple
$\{A,B,C\}$. Certainly, triples with different marks belong to different
$\SSS_m$-orbits. 

Recall that $m \geq 2k+1$. 
First let $|A \cap B| = k-1$, so we may assume 
$$A = \{1,2, \ldots ,k-1,k\},~~B = \{1,2, \ldots ,k-1,k+1\}.$$
For $0 \leq j \leq k-1$, by choosing 
$$C = \{1,2, \ldots ,j,k+2,k+3, \ldots ,2k+1-j\}$$
we get a triple with the mark $(k-1,j,j)$. Similarly, for $1 \leq j \leq k-1$, 
by choosing 
$$C = \{1,2, \ldots ,j-1,k,k+2,k+3, \ldots ,2k+1-j\}$$
we get a triple with the mark $(k-1,j,j-1)$. 

Next we consider the case $|A \cap B| = k-2$, say 
$$A = \{1,2, \ldots ,k-2,k-1,k\},~~B = \{1,2, \ldots ,k-2,k+1,k+2\}.$$
For $1 \leq j \leq k-2$, by choosing 
$$C = \{1,2, \ldots ,j,k+3,k+4, \ldots ,2k+2-j\}$$
we get a triple with the mark $(k-2,j,j)$. Similarly, for $1 \leq j \leq k-2$, 
by choosing 
$$C = \{1,2, \ldots ,j-1,k,k+2,k+3, \ldots ,2k+2-j\}$$
we get a triple with the mark $(k-2,j,j-1)$. 

We have produced at least $4k-5$ different marks. So we have 
$e_3(X) \geq 3k-5 \geq 4$ if $k \geq 3$. 

Finally, consider the case $k = 2$. Then the triples 
$$\{\{1,2\},\{1,3\},\{2,4\}\},\ \{\{1,2\},\{1,3\},\{4,5\}\},\ \{\{1,2\},\{3,4\},\{5,6\}\}$$ 
have marks
$(1,1,0)$, $(1,0,0)$, and $(0,0,0)$. Furthermore, the triples 
$\{\{1,2\},\{1,3\},\{1,4\}\}$ and $\{\{1,2\},\{1,3\},\{2,3\}\}$ have the same mark $(1,1,1)$, but 
different cardinality of $A \cup B \cup C$, so they produce two more
$\SSS_m$-orbits.  
\end{proof}

To estimate $e_3(X)$ for the second special embedding of $X$ into $\SSS_n$, we need the 
following observation:

\begin{lemma}\label{easy}
Let $Y$ be any group and $\K$ any field. Suppose that $A_1$, $A_2$, $B_1$,
$B_2$ are $\K Y$-modules such that there is an injective 
$f \in \Hom_{\K Y}(A,B)$ with $A = A_1 \oplus A_2$, $B = B_1 \oplus B_2$, and 
$f(A_2) \subseteq B_2$. Then
$$\dim B^Y - \dim A^Y \geq \dim B_1^Y - \dim A_1^Y.$$ 
\end{lemma}

\begin{proof}
Clearly, $A^Y = A_1^Y \oplus A_2^Y$ and $B^Y = B_1^Y \oplus B_2^Y$. Now 
$f$ embeds $A_2^Y$ in $B_2^Y$, whence the claim.
\end{proof}

The following statement is also well known and follows for example from the 
formula for $\rank_\K\eta_{2,3}$ given in \cite{Wilson}:

\begin{lemma}\label{23}
Let $p\neq 2,3$ and $n \geq 4$. Then 
$\eta_{2,3}:M_2 \to M_3$  is injective.
$\hfill \Box$  
\end{lemma}


Now we can prove a reduction lemma to help estimate $e_3(X)$ for the second 
special embedding. 

\begin{lemma}\label{e32}
Let $X = \AAA_m$ and $Y := \SSS_m$ embed into $\SSS_n$ via their actions on 
partitions of $\Delta = \{1,2, \ldots ,m\}$ into $b$ $a$-subsets, 
with $a,b \geq 2$. Suppose that 
$b > s \geq 2$. Set $n' := (sa)!/((a!)^s\cdot s!)$ and let 
$Z := \SSS_{sa}$ embed in $\SSS_{n'}$ via its action on partitions of 
$\Delta' = \{1,2, \ldots ,sa\}$ into $s$ $a$-subsets. Also denote by 
$N_r$ the permutation $\SSS_{n'}$-module corresponding to its action 
on $r$-subsets of $\{1,2, \ldots ,n'\}$. Then
$$\dim M_3^X - \dim M_2^X \geq \dim N_3^Z - \dim N_2^Z.$$ 
\end{lemma}

\begin{proof}
1) Since $M_r$ is a permutation module, $\dim M_r^X$ remains unchanged when
we replace $\F$ by any other field. So throughout this proof we may 
assume $\F = \C$.

With respect to $X$ and $Y$, any pair $\{P,Q\}$ in $\Omega_2$ is a pair of 
two partitions
$$P = \{P_1, \ldots ,P_a\},~~Q = \{Q_1, \ldots ,Q_a\}.$$
Call $\{P,Q\}$ a {\it good pair} if at least $b-s$ subsets $P_i$ 
occur among the $Q_j$. Next, we call a triple 
$\{P,Q,R\} \in \Omega_3$ with $R = \{R_1, \ldots ,R_a\}$ 
a {\it good triple} if all three pairs $\{P,Q\}$, $\{P,R\}$, 
$\{Q,R\}$ are good. We also call $\{P,Q,R\}$ a {\it very good triple} if there 
are at least $b-s$ subsets $P_i$ which occur both among the $Q_j$ and among the 
$R_j$. Then $Y > X$ acts on the following sets: $\Omega_{21}$ of all good pairs, 
$\Omega_{22} := \Omega_2 \setminus \Omega_{21}$, $\Omega_{31}$ of all good
triples, and $\Omega_{32} :=  \Omega_3 \setminus \Omega_{31}$. Thus 
$$A_1 := \C\Omega_{21},~~A_2 := \C\Omega_{22},~~
  B_1 := \C\Omega_{31},~~B_2 := \C\Omega_{32}$$
are $\C Y$-submodules of $M_2 = A_1 \oplus A_2$ and $M_3 = B_1 \oplus B_2$. 
Note that $Y > X$ also acts on the set $\Omega_{311}$ of {\it very good} triples,
and so $B_{11} := \C\Omega_{311}$ is an $X$-submodule of $B_1$. Certainly,
$B_{11}^X \subseteq B_1^X$.

Since $\Char(\C) = 0$, $\eta_{2,3}$ is injective by Lemma \ref{23}. Next,
if $\{P,Q\} \in \Omega_{22}$, then
$$\eta_{2,3}(\{P,Q\}) = \sum_{R \neq P,Q}\{P,Q,R\},$$
where all occuring triples $\{P,Q,R\}$ are not good (since $\{P,Q\}$ is not
good). Thus $\eta_{2,3}(\{P,Q\}) \in B_2$. We have shown that 
$\eta_{2,3}(A_2) \subseteq B_2$. Applying Lemma \ref{easy} to the homomorphism
$f = \eta_{2,3}$, we see that
$$\dim M_3^X - \dim M_2^X \geq \dim B_1^X - \dim A_1^X \geq 
  \dim B_{11}^X - \dim A_1^X.$$

Recall that $b > s$. Hence, for any good pair $\{P,Q\}$, $P$ and $Q$ have at 
least one common $a$-subset $P_1$, and since $a \geq 2$, some 
transposition $(i,j) \in Y \setminus X$ fixes both $P$ and $Q$. It follows that 
$X$ and $Y$ have the same orbits on good pairs. Similarly, $X$ and $Y$ have 
the same orbits on {\it very good} triples. Thus 
$$\dim B_{11}^X - \dim A_1^X = \dim B_{11}^Y - \dim A_1^Y.$$
  
\smallskip
2) It remains to prove that 
\begin{equation}\label{dim23}
  \dim B_{11}^Y - \dim A_1^Y = \dim N_3^Z - \dim N_2^Z.
\end{equation}
To do so, we will count the $X$-orbits on good pairs and very good triples.

Let $r \in \{2,3\}$. 
Suppose the good pairs (if $r=2$), respectively the very good triples (if
$r = 3$), $\{P^1, \ldots ,P^r\}$ and $\{Q^1, \ldots ,Q^r\}$ belong to the same 
$X$-orbit. Then, without loss we may assume that
\begin{equation}\label{pq}
  P^i = \{R_1, \ldots ,R_{b-s},P^i_1, \ldots ,P^i_s\},~~
  Q^i = \{R_1, \ldots ,R_{b-s},Q^i_1, \ldots ,Q^i_s\}
\end{equation}
for $1 \leq i \leq r$. Denote 
$$\bar{P}^i := \{P^i_1, \ldots ,P^i_s\},~~
  \bar{Q}^i := \{Q^i_1, \ldots ,Q^i_s\}$$
for $1 \leq i \leq r$.

By assumption, there is some $\sigma \in \SSS_m$ sending 
$\{P^1, \ldots ,P^r\}$ to $\{Q^1, \ldots ,Q^r\}$.
Let $t$ be the number of common $a$-subsets that occur in all 
$P^1, \ldots ,P^r$. Then $t$ is also the number of common $a$-subsets that 
occur in all $Q^1, \ldots ,Q^r$, and $t \geq b-s$.

Consider the case $t = b-s$. Then the $t$ common $a$-subsets among all
$P^i$ are precisely $R_1, \ldots ,R_{b-s}$, and similarly, the $t$ common 
$a$-subsets among all $Q^i$ are precisely $R_1, \ldots ,R_{b-s}$. It follows that
$\sigma$ acts on the set $\{R_1, \ldots ,R_{b-s}\}$, and preserves the set
$$\cup^s_{j=1}P^i_j = \cup^s_{j=1}Q^i_j$$
which can be identified with $\Delta'$. Now we can write 
$\sigma = \mu\tau$, where 
$$\mu \in \SSS_{m-sa} = \Sym(\cup^{b-s}_{j=1}R_j)$$
acts on the set $\{R_1, \ldots ,R_{b-s}\}$, and $\tau \in Z = \Sym(\Delta')$ sends
$\{\bar{P}^1, \ldots ,\bar{P}^r\}$ to $\{\bar{Q}^1, \ldots ,\bar{Q}^r\}$. 
Thus, the two pairs, respectively triples, 
$\{\bar{P}^1, \ldots ,\bar{P}^r\}$ and $\{\bar{Q}^1, \ldots ,\bar{Q}^r\}$ of 
partitions of $\Delta'$ belong to the same $Z$-orbit.

Next assume that $t > b-s$ and set $v = t+s-b$.
Then the $t$ common $a$-subsets among all
$P^i$ are precisely $R_1, \ldots ,R_{b-s},S_1, \ldots ,S_v$, and similarly, 
the $t$ common $a$-subsets among all $Q^i$ are precisely 
$R_1, \ldots ,R_{b-s},T_1, \ldots ,T_v$, for some
$a$-subsets $S_j$ and $T_j$ of $\Delta$. It follows that
$\sigma$ sends $\{R_1, \ldots ,R_{b-s},S_1, \ldots ,S_v\}$ to  
$\{R_1, \ldots ,R_{b-s},T_1, \ldots ,T_v\}$,
and $\Sigma$ to $\Theta$, where 
$$\Sigma = \Delta \setminus (\bigcup^{b-s}_{j=1}R_j \cup \bigcup^v_{j=1}S_j),~~
  \Theta = \Delta \setminus (\bigcup^{b-s}_{j=1}R_j \cup \bigcup^v_{j=1}T_j).$$
Set $R'_j := R_j$ for $1 \leq j \leq b-s$. Also set $R_{b-s+j} := S_j$ and 
$R'_{b-s+j} := T_j$ for $1 \leq j \leq v$.
Then there is a permutation $\pi \in \SSS_t$ such that 
$\sigma(R_j)=R'_{\pi(j)}$. Now we can find $\gamma \in \SSS_m = \Sym(\Delta)$
such that   
$$\gamma_\Sigma = 1_\Sigma,~~\gamma(R_j) = R_{\pi^{-1}(j)}.$$
Clearly, $\sigma\gamma$ sends $R_j$ to $R_j$ for 
$1 \leq j \leq b-s$ (and $S_j$ to $T_j$), and sends $\{P^1, \ldots ,P^r\}$ to 
$\{Q^1, \ldots ,Q^r\}$. Now we can repeat the argument of the preceding case 
$t = b-s$ to show that the two pairs, respectively triples, 
$\{\bar{P}^1, \ldots ,\bar{P}^r\}$ and $\{\bar{Q}^1, \ldots ,\bar{Q}^r\}$ of 
partitions of $\Delta'$ belong to the same $Z$-orbit.   
 
Conversely, it is obvious that if $\{\bar{P}^1, \ldots ,\bar{P}^r\}$ and 
$\{\bar{Q}^1, \ldots ,\bar{Q}^r\}$ belong to the same $Z$-orbit, then 
$\{P^1, \ldots ,P^r\}$ and $\{Q^1, \ldots ,Q^r\}$, defined as in (\ref{pq}), 
belong to the same $Y$-orbit. We have therefore proved 
that $\dim B_{11}^Y = \dim N_3^Z$ and $\dim A_1^Y = \dim N_2^Z$, whence 
(\ref{dim23}) holds.   
\end{proof}

\begin{theorem}\label{bound} 
Let $p=2$ and $n$ be even. Suppose that $n > m \geq 11$ and $X\cong\AAA_m$ embeds into $\AAA_n$ via its actions on
subsets or partitions of $\{1,2, \ldots ,m\}$.  
Then $f_2(X)\geq 3$ and $e_3(X)\geq h(X)+2$, unless $\AAA_m$ embeds into $\AAA_n$ via its 
action on $2$-subsets of $\{1,2, \ldots ,m\}$.
\end{theorem}

\begin{proof} 
1) The inequality $e_2(X) \geq 2$, and hence $f_2(X)\geq 3$, is proved in Lemma \ref{e2}.

If $\AAA_m$ embeds into $\AAA_n$ via its action of $k$-subsets of 
$\Delta$ with $2 < k < m/2$, then $h(X) \leq 2$ by Lemma \ref{h1}, and 
$e_3(X) \geq 4$ by Lemma \ref{e31}. Thus $e_3(X) \geq h(X)+2$ as stated.

\smallskip
2) From now on, we assume that $\AAA_m$ embeds in $\AAA_n$ via its action on
partitions of $\Delta$ into $b$ $a$-subsets, $a,b \geq 2$. By Lemma 
\ref{h1}, either $h(X) \leq 2$, or $h(X) = 3$ and $b = 2|(m/2)$. 

Here we consider the case $a \geq 5$ and show that $e_3(X) \geq 5 \geq h(X)+2$ 
in this case. Applying Lemma \ref{e32} with $s = 2$, we 
are reduced to prove that $\dim M_3^Y - \dim M_2^Y \geq 5$ for $b=2$ and 
$Y = \SSS_m$. In this base case $b = 2$, each partition is a pair 
$$[A] : = \{A,\Delta \setminus A\}$$ 
with $|A| = a = |\Delta|/2$. Hence the 
$Y$-orbits on pairs of partitions $\{[A],[B]\}$ are labeled by 
$1 \leq j \leq t :=  \lfloor a/2 \rfloor$, where $|A \cap B| = a-j$. 
In particular, $\dim M_2^Y = t$. 

So we need to produce at least
$(t+5)$ $Y$-orbits on triples of partitions $\{[A],[B],[C]\} \in \Omega_3$. As in 
the proof of Lemma \ref{e31}, we will label $A,B,C$ so that
\begin{equation}\label{mark}
  |A \cap B| \geq |A \cap C| \geq |B \cap C| \geq a/2
\end{equation}
and call $(|A \cap B|,|A \cap C|,|B \cap C|)$ the {\it mark} of the triple
$\{A,B,C\}$. Certainly, triples with different marks belong to different
$\SSS_m$-orbits. 

First let $|A \cap B| = a-1$, so we may assume 
$$A = \{1,2, \ldots ,a-1,a\},~~B = \{1,2, \ldots ,a-1,a+1\}.$$
For $1 \leq j \leq t$, by choosing 
$$C = \{1,2, \ldots ,a-j,a+2,a+3, \ldots ,a+j+1\}$$
we get a triple with the mark $(a-1,a-j,a-j)$. In fact, for $j = 1$, we have 
two choices for $C$:
$$\{1,2, \ldots ,a-1,a+2\}, \mbox{ and }\{1,2, \ldots ,a-2,a,a+1\}$$
which lead to two different $\SSS_m$-orbits with mark $(a-1,a-1,a-1)$ for
any $a \geq 4$ (since 
$|A \cap B \cap C| = a-1$ for the first choice and $|A \cap B \cap C| = a-2$
for the second choice). Similarly, for 
$1 \leq j \leq t-1$, by choosing 
$$C = \{1,2, \ldots ,a-j-1,a,a+2,a+3, \ldots ,a+j+1\}$$
we get a triple with the mark $(a-1,a-j,a-j-1)$. 

Suppose in addition that $a \geq 6$. Then we choose 
$$A = \{1,2, \ldots ,a-2,a-1,a\},~~B = \{1,2, \ldots ,a-2,a+1,a+2\}.$$
Taking 
$$C = \{1,2, \ldots ,a-2,a+3,a+4\} \mbox{ or }\{1,2, \ldots ,a-3,a+3,a+4,a+5\}$$
we get triples with the mark $(a-2,a-2,a-2)$ and $(a-2,a-3,a-3)$. We have 
produced at least $(t+1)+(t-1)+2 \geq t+5$ orbits on triples, as desired.

Next assume that $a = 5$, and so $t = 2$. We have already produced $4$ orbits
with marks $(4,4,4)$ (two orbits), $(4,4,3)$, and $(4,3,3)$. 
We can also exhibits $3$ more orbits with 
$$\{A,B,C \} = 
  \{12345,12346,12578\},~\{12345,12367,12389\},~\{12345,12367,12589\}$$
and mark $(4,3,2)$, $(3,3,3)$, and $(3,3,2)$, respectively. (Note that 
the last mark deviates from the convention (\ref{mark}), but it does not 
cause any problem since $a = 5$, as one can check.) 

For the next part of the proof, we also consider the case 
$a = 4$, and so $s = 2$. We have already produced $4$ orbits
with marks $(3,3,3)$ (two orbits), $(3,3,2)$, and $(3,2,2)$. Next, the triples with
$$\{P,Q,R\} = \{1234,1256,3456\} \mbox{ and }\{1234,1256,1357\}$$
have the same mark $(2,2,2)$, but belong to different orbits. (Indeed, 
the former satisfies the identity $R = P+Q$, but for the latter
no member of $[R]$ can be the sum of a member of $[P]$ with a member of
$[R]$.) Thus we get at least $6$ orbits on triples of partitions.
(One can show by \cite{GAP} that the number of orbits on triples is 
indeed $6$ for both $\AAA_m$ and $\SSS_m$.) 

\smallskip
3) It remains to consider the case $2 \leq a \leq 4$. Since $m \geq 11$, 
we must have that $b \geq 3$, and so $h \leq 2$ by Lemma \ref{h1}.
For $a = 2$, $3$, or $4$, we set $s = 5$, $3$, or $2$, respectively.
Applying Lemma \ref{e32}, we are reduced to prove that 
$e' := \dim M_3^Y - \dim M_2^Y \geq 4$ for $b=s$ and $Y = \SSS_{sa}$. 
This has been done in 2) for $(a,s) = (4,2)$. 
Using (\ref{fixed}) and \cite{GAP},
one can check that $e' = 35$ for $(a,s) = (3,3)$. Finally, assume that 
$(a,s) = (2,5)$. Using (\ref{fixed}) and \cite{GAP}, we see that 
the number of $\AAA_{10}$-orbits on $2$-subsets, respectively on $3$-subsets of
$\Omega = \{1,2, \ldots, 945\}$ is $6$, respectively $139$. Since 
$\AAA_{10}$ has index $2$ in $Y = \SSS_{10}$, it follows that 
$\dim M_2^Y \leq 6$ and $\dim M_3^Y \geq 139/2$, whence $e' \geq 64$. 
\end{proof}

Now we can prove the main result concerning the special embeddings of 
$\AAA_m$ into $\AAA_n$:

\begin{theorem}\label{main-alt1}
Let $X = \AAA_m$ be embedded in $\AAA_n$ via its actions on partitions or 
on $k$-subsets  of $\{1,2, \ldots ,m\}$ with $2 \leq k < m/2$, and $11 \leq m < n$. 
Let $p=2$ and 
let $V$ be any $\F\AAA_n$-module of dimension 
greater than $1$. Then $\Res^{\AAA_n}_X V$ is reducible.   
\end{theorem}

\begin{proof}
Assume for a contradiction that $V$ is irreducible over $X$.
By Lemma \ref{start}, $n \geq m(m-1)/2$. Since $m \geq 11$, we have
$$|X| = \frac{m!}{2} < 2^{\frac{m(m-1)-12}{4}} \leq (2^{\frac{n-6}{4}})^2.$$
The irreducibility of $V$ on $X$ forces that 
$\dim V < \sqrt{|X|} < 2^{(n-6)/4}$. This bound implies by
Proposition \ref{ext1} that $V$ extends to $\SSS_n$. Thus, without loss
we may assume that $V$ is an irreducible $\F\SSS_n$-module. Also,
since $2^{(n-6)/4} < 2^{\lfloor (n-2)/2 \rfloor}$, $V$ cannot be a basic spin module.

Now if $n$ is odd, then the $X$-module $V$ is reducible by 
\cite[Theorem 3.10]{KS1}. Hence $2|n$, and 
$V$ satisfies the conclusion (i) of Theorem \ref{hom}. By Corollary 
\ref{pairs2}, $X$ does {\it not} embed into  
$\AAA_n$ via its action on $2$-subsets of $\{1,2, \ldots ,m\}$. Hence, by Theorem 
\ref{bound}, $f_2(X) \geq 3$ and $e_3(X) \geq h(X)+2$. Clearly, $X$ is a perfect subgroup
and $f_1(X) = 1$. Therefore, the $X$-module $V$ is reducible by Theorem 
\ref{reduction}, a contradiction.
\end{proof}

\section{General embeddings of $\AAA_m$ into $\AAA_n$}
Throughout this section, $X\cong \AAA_m$ is a subgroup of $\AAA_n$, with $n>m\geq 5$. Recall the notation
$\Omega=\{1,2,\dots,n\}$ and $\Delta=\{1,2,\dots,m\}.$

\smallskip
First, we deal with the small cases $5  \leq m \leq 10$:

\begin{lemma}\label{small}
Let $X \cong \AAA_m$ be a transitive subgroup of $\AAA_n$ with $n > m$ and 
$5 \leq m \leq 10$, $V$ be an $\F\AAA_n$-module of dimension $>1$ such that
$\Res^{\AAA_n}_X V$ is irreducible. Then $(m,n) = (5,6)$, $(6,10)$, $(7,15)$, 
or $(8,15)$.
\end{lemma}

\begin{proof}
1) First we consider the case where $V$ is not the heart of the natural permutation 
module of $Y := \AAA_n$. Suppose for instance that $m = 10$. The assumptions that
$X$ is a transitive subgroup of $Y$ and $n > m$ imply by \cite{Atlas} that $n \geq 45$. 
This in turn implies by \cite[Lemma 6.1]{GT1} that 
$\dim V \geq (n^2-5n+2)/2 = 901 > 567 \geq b_\F(X)$, where $b_\F(X)$ denotes the largest degree of irreducible $\F X$-representations. It follows that $\Res^{\AAA_n}_X V$ is 
reducible. The same arguments apply to the case $m = 9$, where we have 
$n \geq 36$ and $b_\F(X) \leq 216$.  

Suppose now that $m = 8$. According to \cite{Atlas} and \cite{JLPW}, $n = 15$ or 
$n \geq 28$; furthermore, $b_\F(X) \leq 70$. As above, $\dim V > b_\F(X)$ if
$n \geq 28$. It follows that $n = 15$. The same arguments apply to the case 
$5 \leq m \leq 7$. (In fact, one can show that we must have either $(m,n) = (5,6)$
or $(m,n,p,\dim V) = (8,15,2,64)$.)  

\smallskip
2) Now we may assume that $V$ is the heart of the natural permutation 
module $\F\Omega$ of 
$Y$. By the assumption, $X$ acts transitively
on $\Omega$. If this action is not doubly transitive, then the restriction of the 
permutation character of $\C\Omega$ is $1_X + \alpha + \beta$ where $\alpha$, $\beta$
are (not necessarily irreducible) $X$-characters of degree $> 1$. 
Since the Brauer character of $\F\Omega$ is $\varphi + e \cdot 1_Y$ with  
$e \in \{ 1,2 \}$ and $\varphi \in \IBR_p(Y)$,  it follows that $V$ is reducible over $X$. 
Thus $X$ acts doubly transitively on $\Omega$. Now we can read off the possible 
value of $n$ using \cite{Atlas}. 
\end{proof}

\begin{remark} \label{RSmallCases} 
{\rm 
The special cases $(m,n)$ listed in the statement of Lemma~\ref{small} are indeed exceptional. For these pairs, we can embed $\AAA_m$ into $\AAA_n$  
so that $\AAA_m$ acts doubly transitively on $\Omega$. Take $V$ to be the heart of the
permutation module $\C\Omega$. It is well known that any doubly transitive group  is irreducible on $V$. This is also almost always true for $p>0$ --- see \cite{Mortimer}.  

}
\end{remark}

We will need the following group-theoretic result:

\begin{lemma}\label{index}
Let $X \cong \AAA_m$ be a transitive subgroup of $\AAA_n$ with $n > m \geq 11$.
Then $n \geq m(m-1)/2$ and $X$ cannot act doubly transitively on $\Omega$. If $X$ is imprimitive, then $n \geq m(m-1)$. 
If $X$ is primitive, then one of the following statements holds:
\begin{enumerate}
\item[{\rm (i)}] $X$ acts on $\Omega$ via its action on $k$-subsets of $\Delta$ 
with $2 \leq k < m/2$ or on partitions of $\Delta$, and 
$n \geq m(m-1)/2$. 
\item[{\rm (ii)}] $n > |X|^{3/10} > m(m-1)$.
\end{enumerate}
\end{lemma}

\begin{proof}
From the classification of doubly transitive 
permutation representations of $\AAA_m$ \cite{Ma} and the assumption  
$m \geq 11$, it follows that $\AAA_m$ has no doubly transitive permutation
representation of degree $n > m$. Next, 
consider the transitive action of $X$ on $\Omega$ and 
let $X_1$ be a point stabilizer of this action. Then $[X:X_1] = n > m$.
We can find a maximal subgroup $M$ of $X$ containing $X_1$. 

Now we consider the action of $M$ on 
$\Delta$. If this action is intransitive, then 
by maximality of $M$, $M$ is the stabilizer in $X$ of a $k$-subset
of $\Delta$ for some $1 \leq k < m/2$. On the other hand, if this action
is transitive but imprimitive, then $M$ is the stabilizer in $X$ of 
a partition of $\Delta$ into $b$ $a$-subsets, with $a,b \geq 2$ and 
$ab = m$. By Lemma \ref{start}, in either case 
we have $n \geq [X:M] \geq m(m-1)/2$ (whence 
$[X:X_1] \geq m(m-1)$ if $X$ is imprimitive),
unless $M$ is intransitive and $k=1$. In this exceptional case,
$M = \AAA_{m-1}$ and $[X:X_1] = n > m = [X:M]$. Thus $X_1$ is a proper
subgroup of $M$, and so, as is well known, $[M:X_1] \geq m-1$ and 
$n \geq m(m-1)$. 

Suppose now that $M$ acts primitively on $\Delta$. By Bochert's 
Theorem 14.2 of \cite{Wie}, $[\SSS_m:M] \geq \lfloor (m+1)/2 \rfloor !$. 
Since $m \geq 11$, we have 
$\frac{\lfloor (m+1)/2 \rfloor !}{2} > \left(\frac{m!}{2}\right)^{3/10}.$
Therefore $n \geq [X:M] > |X|^{3/10} > m(m-1)$ as $m \geq 11$.

We have proved the bound $n \geq m(m-1)/2$ and also that 
$n \geq m(m-1)$ if $X$ is imprimitive. Assume now that $X$ is primitive.
Then $X_1 = M$, and the above analysis yields the final claim of the lemma.
\end{proof}

\begin{theorem}\label{main-alt2}
Let $X \cong \AAA_m$ be a primitive subgroup of $\AAA_n$ with $n > m \geq 9$.
Let $p=2$ and 
let $V$ be an 
$\F\AAA_n$-module of dimension $>1$. Then the restriction $\Res^{\AAA_n}_X V$ is 
reducible.
\end{theorem}

\begin{proof}
By Lemma~\ref{small}, we may assume that $m\geq 11$. 
Assume for a contradiction that $V$ is irreducible over $X$.
We apply Lemma \ref{index}. In the case (i) of Lemma \ref{index}  
we are done by Theorem \ref{main-alt1}.

Now suppose that the case (ii) of Lemma \ref{index} holds. Then 
$|X| < n^{10/3}$ and $n \geq 155$ as
$m \geq 11$. Since $V$ is irreducible over $X$, we must have that 
$$\dim V < n^{5/3} < \frac{1}{2}(n-1)(n-2).$$ 
In particular, $\dim V < 2^{(n-6)/4}$ (as $n \geq 155$), 
whence $V$ extends to $\SSS_n$ by Proposition \ref{ext1}. We denote the 
(unique up to isomorphism) extension of $V$ to $\SSS_n$ by the same letter~$V$. 

By Lemma \ref{index}, the action of $X = \AAA_m$ on 
$\Omega$ is not doubly transitive. Hence,
if $n$ is odd then, we are done by appealing to \cite[Theorem 3.10]{KS1}. 
Assume that $2|n \geq 155$. 
Then the condition $1 < \dim V < (n-1)(n-2)/2$ for the irreducible 
$\F\SSS_n$-module $V$ implies by \cite[Theorem 7]{J2} that 
$V \cong D^{(n-k,k)}$ 
with $k = 1$ or $2$. As before, let $1+\alpha$ denote the (complex) 
permutation character of $\SSS_n$ on $\Omega$. 

Suppose $k = 1$. Then $V = D^{(n-1,1)}$ is irreducible over $X$. Denoting by
$\alpha^0$ the restriction of $\alpha$ to $2'$-elements in $\SSS_n$, we see
that the Brauer character of $V$ is $\alpha^0-1$. Furthermore, $X$ is perfect
and $[\Res^{\SSS_n}_X \alpha,1_X]_X = 0$ by transitivity of $X$ on $\Omega$. Hence
the irreducibility of $V$ over $X$ forces that $\alpha$ is also irreducible over $X$.
In other words, $X$ acts doubly transitively on $\Omega$, a contradiction.

Assume now that $k = 2$. Then $\wedge^2(D^{(n-1,1)})$ has a composition
series with composition factors $V$ (once), and $\triv$ (once if $4|n$ and twice
if $n \equiv 2 \pmod 4$). The same is true for $\wedge^2(D^{(n-1,1)})$ considered as an $X$-module. 
Suppose in addition that $D^{(n-1,1)}$ is reducible over $X$. 
Then we can write the Brauer character of the $\F X$-module 
$\Res^{\SSS_n}_X D^{(n-1,1)}$ as 
$\beta+\gamma$, where $\beta$ and $\gamma$ are Brauer characters of $X$,
$\beta(1) \geq 1$ and $\gamma(1) \geq 3$. Note that
$$\wedge^2(\beta+\gamma) = \wedge^2(\beta) + \wedge^2(\gamma) + \beta\gamma.$$
It follows that the Brauer character of the $X$-module $\wedge^2(D^{(n-1,1)})$ 
is the sum of three Brauer characters, with at least two of degree $\geq 3$.
This contradicts the aforementioned composition structure of the $X$-module 
$\wedge^2(D^{(n-1,1)})$. Thus $D^{(n-1,1)}$ is irreducible over $X$. But then, arguing as
in the previous paragraph, we again arrive at the contradiction that $X$ acts doubly transitively on 
$\Omega$.   
\end{proof}

The case $p\neq 2$ can be done using results of 
\cite{KS1}, \cite{KS2} and Proposition \ref{ext2}: 

\begin{proposition}\label{alt3}
Let $X \cong \AAA_m$ be a transitive subgroup of $\AAA_n$ with $n > m \geq 9$.
Let $p \neq 2$ and  $V$ be an $\F\AAA_n$-module of dimension $>1$. 
Then $\Res^{\AAA_n}_X V$ is reducible.
\end{proposition}

\begin{proof}
By Lemma~\ref{small}, we may assume that $m\geq 11$. 
Assume for a contradiction that $\Res^{\AAA_n}_X V$ is irreducible. 
By Lemma \ref{index}, $X$ is not doubly transitive, and $n \geq m(m-1)/2$.
If $V$ is extendible to $\SSS_n$, we can apply \cite[Main Theorem]{KS1} to get 
a contradiction. Suppose $V$ is not extendible to $\SSS_n$. If $p \neq 3$, we again 
arrive at a contradiction by using \cite[Main Theorem]{KS2}. 
So we may assume that $p=3$. By Proposition \ref{ext2}(ii) we have that
$$\dim V \geq 2^{\frac{n-8}{4}} \geq 2^{\frac{m^2-m-16}{8}} =:c_m.$$
Now, if $m \geq 12$ then $c_m > \sqrt{m!/2} = \sqrt{|X|}$. If 
$m = 11$, then $c_m > 3444$ whereas the largest degree of complex irreducible representations of $X = \AAA_{11}$ is $2310$. 
In either case, $X$ cannot be irreducible on $V$. 
\end{proof}

Now Theorem \ref{main} follows immediately from Theorem \ref{main-alt2} and
Proposition \ref{alt3}.

Note that for $p\neq 2$, Proposition~\ref{alt3} shows that any proper transitive  subgroup $X\cong\AAA_m$ of $\AAA_n$ acts reducibly on all non-trivial modules over $\F\AAA_n$ (provided $m\geq 9$). For $p=2$, if $X$ is primitive, the same result holds by Theorem~\ref{main-alt2}. Now we handle imprimitive embeddings of $\AAA_m$ into $\AAA_n$ for $p=2$. Here our result is a little weaker: 

\begin{proposition}\label{alt22}
Let $X \cong \AAA_m$ be a (transitive) imprimitive subgroup of $\AAA_n$ with 
$n > m \geq 9$. Let $p=2$ 
and let $V$ be an irreducible $\F\AAA_n$-module of dimension $>1$. Then 
the $X$-module $\Res^{\AAA_n}_X V$ cannot be primitive irreducible.
\end{proposition}

\begin{proof}
By Lemma~\ref{small}, we may assume that $m\geq 11$. 
Assume for a contradiction $V$ is irreducible and primitive over $X$.
Consider the action of $X$ as a subgroup of $\AAA_n$ on $\Omega$
and let $X_1$ be a point stabilizer. Since $X$ is imprimitive, there is 
a maximal subgroup $M > X_1$ of $X$. Now $b := [X:M] \geq m \geq 11$, and 
we may assume that $X < Y := (\SSS_a \wr \SSS_b) \cap \AAA_n$ for $a := n/b > 1$. 
By Lemma \ref{index}, $n \geq m(m-1)$. If $a = 2$ or $4$, then $O_2(Y) > 1$, and so 
$V$ cannot be irreducible over $Y > X$, a contradiction. So we have $a \geq 3$
and $a \neq 4$.

Let $\lambda = (\lambda_1,\lambda_2,\dots)$ be a $2$-regular partition
of $n$ such that $V$ is a simple submodule of $\Res^{\SSS_n}_{\AAA_n} D^\lambda$. 
Observe 
that $m! < (m/2)^m$ for $m \geq 6$.  
Since $\Res^{\AAA_n}_X V$ is irreducible, by \cite[Theorem 5.1]{GLT} we have
$$2^{\frac{n-\lambda_1}{2}-1} \leq \dim V \leq \sqrt{|\AAA_m|} 
  < \sqrt{\frac{m^m}{2^{m+1}}},$$
whence $n-\lambda_1 < m\log_2m-m+1$. This in turn implies that 
$2\lambda_1 -n \geq n/m$. Indeed, otherwise we would have 
$\lambda_1 < n/2+n/2m$ and so 
$$n-\lambda_1 > \frac{n}{2}-\frac{n}{2m} = n \cdot \frac{m-1}{2m}
  \geq \frac{(m-1)^2}{2} > m\log_2m-m+1,$$
a contradiction.   

We have shown that $2\lambda_1-n \geq n/m \geq n/b = a \geq 3$. Also,
$n-\lambda_1 \geq 1$ since $\dim V > 1$. Applying Proposition \ref{two}, 
we see that the restriction of $D^\lambda$ to a natural subgroup $\SSS_a$ of 
$\SSS_n$ affords both $\triv$ and $D^{(a-1,1)}$ as composition factors.
If $a \geq 5$, then $\dim D^{(a-1,1)} \geq 4$ and so any irreducible
summand of the $\AAA_a$-module $D^{(a-1,1)}$ has dimension $\geq 2$ (in fact 
$D^{(a-1,1)}$ is irreducible over $\AAA_a$). On the other hand, if $a = 3$,
then $D^{(a-1,1)}$ splits into a direct sum of two non-trivial irreducible 
$\AAA_a$-submodules. Thus $\Res^{\AAA_n}_{\AAA_a} V$ affords both the trivial module and 
also another non-trivial irreducible module as composition factors. It follows
that $\Res^{\AAA_n}_Z V$ cannot be homogeneous for 
$Z := \AAA_a \times \ldots \times \AAA_a = (\AAA_a)^b \lhd Y$. But this is a contradiction,
as $Z \lhd Y$ and the $Y$-module $V$ is primitive.
\end{proof}


The following lemma deals with tensor indecomposable irreducible representations of~$\AAA_m$. These were studied extensively in \cite{Z, BKTens0,BKTensp,GK,GJ,BKAltTens}. 

\begin{lemma}\label{tensor}
Let $X \cong \AAA_m$ be a subgroup of $\AAA_n$ with $n \geq m \geq 5$.
Let  $V$ be an irreducible $\F\AAA_n$-module of dimension $>1$. 
Assume that the $X$-module $\Res^{\AAA_n}_X V$ is irreducible and tensor indecomposable.
Then one of the following holds:
\begin{enumerate}
\item[{\rm (i)}] $X$ is a transitive subgroup of $\AAA_n$.
\item[{\rm (ii)}] There is some $t \in \{1,2\}$ such that $X$ fixes $t$ points and
acts transitively on $n-t$ remaining points of $\Omega$. Furthermore,
$V$ is irreducible over a natural subgroup $\AAA_{n-t}$ of $\AAA_n$. 
\end{enumerate}
\end{lemma}

\begin{proof}
Suppose that $X$ acts intransitively on $\Omega$. Each
non-trivial orbit of $X$ has at least $m \geq 5$ points. If $X$ has at least 
two non-trivial orbits on $X$, then we may assume that there is some 
$5 \leq k \leq n/2$ such that $X$ acts non-trivially on both 
$\Omega^{(1)} := \{1,2, \ldots ,k\}$ and $\Omega^{(2)} := \{k+1, \ldots ,n\}$. 
These actions induce embeddings $\pi_i~:~X \to \Alt(\Omega^{(i)})$, 
namely, $g \in X$ acts on $\Omega^{(i)}$ as $\pi_i(g)$.
Now, $X < Y := \AAA_k \times \AAA_{n-k}$ and
$\Res^{\AAA_n}_Y V$ is irreducible. Hence $V{|}_Y \cong V_1 \boxtimes V_2$ is an outer tensor 
product of irreducible modules $V_1$ over $\AAA_k$ and $V_2$ over  $\AAA_{m-k}$. Note 
that $\dim V_1 > 1$, as otherwise the $3$-cycles in $\AAA_k$ would act 
trivially on $V$. Similarly $\dim V_2 > 1$. It follows that
$\Res^{\AAA_n}_X V = U_1 \otimes U_2$, where 
$U_i \cong \Res^{\AAA_{k_i}}_{\pi_i(X)} V_i$ for $i=1,2$ and
$(k_1,k_2) := (k,n-k)$. In particular,
$\Res^{\AAA_n}_X V$ is tensor decomposable, contrary to the assumption.      

Thus $X$ has only one non-trivial orbit on $\Omega$, say of length $n-t$
for some $t > 0$, and $X$ fixes the remaining $t$ points. Furthermore, 
$t \leq 2$, as otherwise $X$ is centralized by a $3$-cycle and so cannot
act irreducibly on $V$.
\end{proof}

Finally, in connection with Lemma \ref{tensor}, we bound the number
of extensions of irreducible representations of $\AAA_{n-1}$ to $\AAA_n$. For this type of 
question, it is convenient to use the following simple observation:

\begin{lemma}\label{iso}
Let $Y$ be a subgroup of $X$ and $U$ be a finite-dimensional vector space over
$\F$. Let $\Psi~:~Y \to GL(U)$ be an irreducible representation. Suppose that 
$\Phi_i~:~X \to GL(U)$, $i = 1,2$, are two isomorphic representations of $X$ with
$\Res^X_Y(\Phi_i) = \Psi$. Then in fact $\Phi_1 = \Phi_2$.
\end{lemma}

\begin{proof}
Fix a basis of the vector space $U$ and write $\Psi$ and $\Phi_i$ as matrix representations
wit respect to this basis. By the assumption, there is some invertible matrix $A$ such that 
$\Phi_2(x) = A\Phi_1(x)A^{-1}$ for  all $x \in X$. Since both $\Phi_i$ extend $\Psi$,
we have that $\Psi(y) = A\Psi(y)A^{-1}$ for all $y \in Y$. By Schur's lemma, $A$ is scalar, and so $\Phi_2 = \Phi_1$.
\end{proof}

Now we deal with the aforementioned question for symmetric groups:

\begin{lemma}\label{ext-sym}
Assume $n \geq 3$ and fix a natural embedding of $\SSS_{n-1}$ into $\SSS_n$. 
Then every irreducible $\F\SSS_{n-1}$-representation 
$\SSS_{n-1} \to GL(U)$ has at most one extension to $\SSS_n$.
\end{lemma}

\begin{proof}
Consider $U$ as an $\F\SSS_{n-1}$-module, and suppose that there are 
two distinct $p$-regular partitions $\lambda,\mu \vdash n$ such that 
$$\Res^{\SSS_n}_{\SSS_{n-1}}(D^\lambda) \cong  
    \Res^{\SSS_n}_{\SSS_{n-1}}(D^\mu) \cong U.$$ 
Since 
$\soc(\Res^{\SSS_n}_{\SSS_{n-1}}(D^\lambda)) \cong  
  \soc(\Res^{\SSS_n}_{\SSS_{n-1}}(D^\mu))$, 
by \cite[Corollary 5.3]{K4}, precisely one of the modules 
$\Res^{\SSS_n}_{\SSS_{n-1}}(D^\lambda)$, 
$\Res^{\SSS_n}_{\SSS_{n-1}}(D^\mu)$ is reducible, a contradiction. Now apply 
Lemma \ref{iso}.
\end{proof}

\begin{proposition}\label{ext-alt}
Assume $n \geq 5$ and fix a natural embedding of $\AAA_{n-1}$ into $\AAA_n$. 
Then every irreducible $\F\AAA_{n-1}$-representation 
$\AAA_{n-1} \to GL(U)$ has at most three distinct extensions to $\AAA_n$.
\end{proposition}

\begin{proof}
We assume that $U$ extends to an $\F\AAA_n$-module $V_1$ and find all possible 
other extensions $V_2$ of $U$ to $\AAA_n$. Consider $\AAA_{n-1}$ as the 
derived subgroup of a natural subgroup $\SSS_{n-1}$ of $\SSS_n$. 
By Lemma \ref{iso}, it suffices to bound the number of extensions up to isomorphism.
For each $i = 1,2$, we can find an irreducible $\F\SSS_n$-module $W_i$ such that 
$V_i \hookrightarrow \soc(\Res^{\SSS_n}_{\AAA_n}(W_i))$, and an irreducible 
$\F\SSS_{n-1}$-module  $T$ such that 
$U \hookrightarrow \soc(\Res^{\SSS_{n-1}}_{\AAA_{n-1}}(T))$. We also fix a transposition
$g \in \SSS_{n-1} \setminus \AAA_{n-1}$, and distinguish 
the following cases.

\smallskip
{\bf Case I}: $U$ is not $\SSS_{n-1}$-invariant. Then for $i = 1,2$ 
 we have that
$$\Res^{\AAA_{n}}_{\AAA_{n-1}}(V_i^g) = U^g \not\cong U = 
    \Res^{\AAA_{n}}_{\AAA_{n-1}}(V_i),$$
and so $V_i$ is not $\SSS_n$-invariant. Hence  
$\Res^{\SSS_{n}}_{\AAA_{n}}(W_i) \cong V_i \oplus V_i^g$, and similarly
$\Res^{\SSS_{n-1}}_{\AAA_{n-1}}(T) \cong U \oplus U^g$. Now
$\Res^{\SSS_{n}}_{\AAA_{n-1}}(W_i) \cong U \oplus U^g \cong 
  \Res^{\SSS_{n-1}}_{\AAA_{n-1}}(T)$. It follows that
$\Res^{\SSS_{n}}_{\SSS_{n-1}}(W_1) \cong T \cong 
  \Res^{\SSS_{n}}_{\SSS_{n-1}}(W_2)$. By Lemma \ref{ext-sym},
$W_1 \cong W_2$, whence $V_2$ is isomorphic to $V_1$ or $V_1^g$. Since 
$\Res^{\AAA_{n}}_{\AAA_{n-1}}(V_1^g) = U^g \not\cong U$, we must have that 
$V_2 \cong V_1$. 

\smallskip
{\bf Case IIa}: $U$ is  $\SSS_{n-1}$-invariant and both $V_1,V_2$ are $\SSS_n$-invariant.
In particular, we have $\Res^{\SSS_{n}}_{\AAA_{n}}(W_i) = V_i$ for $i = 1,2$ and similarly 
$\Res^{\SSS_{n-1}}_{\AAA_{n-1}}(T) = U$. Now if $p = 2$, then the Brauer
character of any extension of $U$ to $\SSS_{n-1}$ is uniquely determined by its restriction
to $\AAA_{n-1}$ which is the Brauer character of $U$. Thus $U$ has a unique extension
to $\SSS_{n-1}$, and so 
$\Res^{\SSS_{n}}_{\SSS_{n-1}}(W_1) \cong \Res^{\SSS_{n}}_{\SSS_{n-1}}(W_2)$. It follows
by Lemma \ref{ext-sym} that $W_1 \cong W_2$ and so $V_1 \cong V_2$.  

Next suppose that $p \neq 2$. Then $U$ has two distinct extensions $T$ and $T^g$
to $\SSS_{n-1}$ and $V_i$ has two distinct extensions $W_i$ and $W_i^g$
to $\SSS_n$. It follows that $\Res^{\SSS_{n}}_{\SSS_{n-1}}(W_i) \in \{T,T^g\}$ and so 
$\Res^{\SSS_{n}}_{\SSS_{n-1}}(W_2) \in \{\Res^{\SSS_{n}}_{\SSS_{n-1}}(W_1), 
    \Res^{\SSS_{n}}_{\SSS_{n-1}}(W_1^g)\}$. 
By Lemma \ref{ext-sym}, $W_2 \cong W_1$ or $W_1^g$, and so $V_2 \cong V_1$.

Thus Case IIa shows that among the extensions of $U$ to 
$\AAA_n$, at most one of them is $\SSS_n$-invariant.

\smallskip
{\bf Case IIb}: $U$ is  $\SSS_{n-1}$-invariant, but neither $V_1$ nor $V_2$ is 
$\SSS_n$-invariant. In this case for $i = 1,2$ we have that
$W_i = \Ind^{\SSS_n}_{\AAA_n}(V_i)$, and so 
$$\Res^{\SSS_{n}}_{\SSS_{n-1}}(W_i) \cong 
    \Ind^{\SSS_{n-1}}_{\AAA_{n-1}}(\Res^{\AAA_{n}}_{\AAA_{n-1}}(V_i))
    \cong \Ind^{\SSS_{n-1}}_{\AAA_{n-1}}(U)$$
is reducible. Also, 
$\soc(\Res^{\SSS_{n}}_{\SSS_{n-1}}(W_1)) \cong 
  \soc(\Res^{\SSS_{n}}_{\SSS_{n-1}}(W_2))$. 
If moreover $W_1 \not\cong W_2$, then the latter 
implies by \cite[Corollary 5.3]{K4} that (precisely) one of the modules 
$\Res^{\SSS_{n}}_{\SSS_{n-1}}(W_i)$, $i = 1,2$, is irreducible, a contradiction. It follows that
$W_1 \cong W_2$ and so $V_2 \in \{V_1, V_1^g\}$ since 
$\Res^{\SSS_{n}}_{\AAA_{n}}(W_1) = V_1 \oplus V_1^g$.   

Thus Case IIb shows that among the extensions of $U$ to 
$\AAA_n$, at most two of them can be non-$\SSS_n$-invariant. Hence in the case where $U$ is $\SSS_{n-1}$-invariant, there are at most three extensions to $\AAA_n$. 
\end{proof}

\section{Rank 3 permutation groups}
To illustrate applicability of Theorem \ref{reduction} to other primitive subgroups of $\SSS_n$,
in this section we consider finite simple classical groups $X$ acting as rank $3$ permutation groups
on $\Omega$, where $(\Omega,X)$ is one of the following:
\begin{enumerate}
\item[\rm (i)] $\Omega$ is the set of $2$-dimensional subspaces of $W = \F_q^d$
and $X = PSL(W) = PSL_d(q)$ with $d \geq 4$;

\item[\rm(ii)] $\Omega$ is the set of singular $1$-dimensional subspaces
of $W = \F_q^d$ and $X = PSp(W)'= PSp_d(q)'$ with $2|d \geq 4$ (note that
$Sp_d(q)' \neq Sp_d(q)$ only when $(d,q) = (4,2)$);

\item[\rm(iii)] $\Omega$ is the set of singular $1$-dimensional subspaces
of $W = \F_q^d$ and $X = P\Omega(W)= P\Omega^\pm_d(q)$ with $d \geq 5$;

\item[\rm(iv)] $\Omega$ is the set of singular $1$-dimensional subspaces
of $W = \F_{q^2}^d$ and $X = PSU(W)= PSU_d(q)$ with $d \geq 4$.
\end{enumerate}
According to the main result of \cite{KaL}, these families account for all the standard rank $3$ 
permutation representations of finite simple classical groups.

\begin{lemma}\label{r3-orbits}
Let $X < \Sym(\Omega) = \SSS_n$, where $(\Omega,X)$ is as listed above, and $2|n$. 
\begin{enumerate}
\item[{\rm (i)}] Assume furthermore that $d \geq 6$ when 
$X = PSp_d(q)$ or $X = PSU_d(q)$, and $d \geq 7$ when $X = P\Omega^\pm_d(q)$. Then 
$f_3(X) \geq 6$. 

\item[{\rm (ii)}] Also, $f_3(X) \geq 5$ if $X = PSp_4(q)$, $PSU_4(q)$, $PSU_5(q)$,
$\Omega_5(q)$, or $P\Omega^\pm_6(q)$.
\end{enumerate}
\end{lemma}
  
\begin{proof}
(a) First we consider the case $X = PSL(W)$ with $W = \langle e_1, \ldots,e_d \rangle_{\F_q}$ 
of dimension $d \geq 4$. Then $X$ has at least six orbits on $\Omega_3$, 
with the following representatives:

$\bullet$ $A = \la e_1,e_2\ra_{\F_q}$, $B= \la e_1,e_3\ra_{\F_q}$, $C = \la e_2,e_3\ra_{\F_q}$
(note that $\dim_{\F_q}(A+B+C) = 3$ and $A \cap B \cap C = 0$ here);

$\bullet$ $A = \la e_1,e_2\ra_{\F_q}$, $B= \la e_1,e_3\ra_{\F_q}$, $C = \la e_1,e_2+e_3\ra_{\F_q}$
(here, $\dim_{\F_q}(A+B+C) = 3$ and $A \cap B \cap C \neq 0$);

$\bullet$ $A = \la e_1,e_2\ra_{\F_q}$, $B= \la e_3,e_4\ra_{\F_q}$, $C = \la e_1+e_3,e_2+e_4\ra_{\F_q}$ (here, $\dim_{\F_q}(A+B+C) = 4$ and 
$\{\dim_{\F_q}(A \cap B),\dim_{\F_q}(A \cap C),\dim_{\F_q}(B \cap C) = \{0,0,0\}$);

$\bullet$ $A = \la e_1,e_2\ra_{\F_q}$, $B= \la e_1,e_3\ra_{\F_q}$, $C = \la e_2+e_3,e_4\ra_{\F_q}$
(here, $\dim_{\F_q}(A+B+C) = 4$ and 
$\{\dim_{\F_q}(A \cap B),\dim_{\F_q}(A \cap C),\dim_{\F_q}(B \cap C) = \{1,0,0\}$);

$\bullet$ $A = \la e_1,e_2\ra_{\F_q}$, $B= \la e_1,e_3\ra_{\F_q}$, $C = \la e_2,e_4\ra_{\F_q}$ 
(here, $\dim_{\F_q}(A+B+C) = 4$ and 
$\{\dim_{\F_q}(A \cap B),\dim_{\F_q}(A \cap C),\dim_{\F_q}(B \cap C) = \{1,1,0\}$);

$\bullet$ $A = \la e_1,e_2\ra_{\F_q}$, $B= \la e_1,e_3\ra_{\F_q}$, $C = \la e_1,e_4\ra_{\F_q}$
(here, $\dim_{\F_q}(A+B+C) = 4$ and $\dim_{\F_q}(A \cap B \cap C) = 1$).

\smallskip
(b) In the remaining cases of (i), the assumption on $d$ implies that $W$ contains a 
non-degenerate $6$-dimensional subspace with hyperbolic (Witt) basis $(e_1,e_2,e_3,f_1,f_2,f_3)$.
Also, since $n = |\Omega|$ is even, $q$ is odd.
We will write any element in $\Omega$ as $\la a \ra$, the $1$-space generated by a vector $a \in W$.
In the $\Omega$-case, for any unordered triple $\pi = \{A,B,C\}$, with 
$A = \la a \ra$, $B = \la b \ra$, $C = \la c \ra \in \Omega$
and $\dim (A+B+C) = 3$, we can associate to it the Gram matrix $\Gamma$ of the bilinear form 
$(\cdot,\cdot)$ written in the basis $(a,b,c)$. Since changing to another basis
of $A+B+C$ changes $\det(\Gamma)$ by a factor which belongs to the 
subgroup $F_0:= \F_q^{\times 2}$ of $F := \F_q^\times$,
we can associate to such $\pi$ a canonical element $\delta := \det(\Gamma)F_0 \in F/F_0$.  
Then the following unordered triples $\{A,B,C\} \in \Omega_3$  belong to disjoint $X$-orbits:

$\bullet$ $A = \la e_1 \ra$, $B= \la e_2 \ra$, $C = \la e_1 + e_2\ra$
(note that $A+B+C$ is a $2$-dimensional totally singular subspace here);

$\bullet$ In the $Sp/SU$-case: 
$A = \la e_1 \ra$, $B= \la f_1 \ra$, $C = \la e_1 + \lambda f_1\ra$,
where $\lambda = 1$ in the $Sp$-case and $\lambda^{q-1} = -1$ in the $SU$-case.
Note that $A+B+C$ is a $2$-dimensional non-degenerate subspace here;

$\bullet$ $A = \la e_1 \ra$, $B= \la e_2 \ra$, $C = \la e_3\ra$
(here, $A+B+C$ is a $3$-dimensional totally singular subspace);

$\bullet$ $A = \la e_1 \ra$, $B= \la f_1 \ra$, $C = \la e_2\ra$.
Note that $U := A+B+C$ is a $3$-dimensional subspace with $C = \rad(U)$;

$\bullet$ $A = \la e_1 \ra$, $B= \la f_1 \ra$, $C = \la e_1+ e_2\ra$.
Note that $U := A+B+C$ is a $3$-dimensional subspace with $\dim\rad(U) = 1$
but $A,B,C \neq \rad(U)$. In fact, in the $Sp$-case, we get one more
triple with the same $A$, $B$, but with $C' = \la e_1+f_1+e_2 \ra$ --
note that $A \perp C$, but no two of $A$, $B$, $C'$ are orthogonal to each other;

$\bullet$ In the $\Omega/SU$-case: 
$A = \la e_1 \ra$, $B= \la f_1 \ra$, $C = \la e_2+ \lambda f_2 + e_1 - \lambda f_1\ra$,
where $\lambda \in F$ in the $\Omega$-case and $\lambda = 1$ in the $SU$-case. 
Here, $W := A+B+C$ is a $3$-dimensional non-degenerate subspace. 
Furthermore, in the $\Omega$-case, we have $\delta = -2\lambda F_0$, whence we can choose 
$\lambda$ so that $\delta = F_0$, respectively $\delta \neq F_0$.

\smallskip
(c) Ignoring the vectors $e_3,f_3$, the arguments in (b) also show that
$f_2(X) \geq 5$ if $X = PSp_4(q)$, $PSU_4(q)$, $PSU_5(q)$, 
$\Omega_5(q)$, or $\Omega^\pm_6(q)$.
\end{proof}  
 
\begin{lemma}\label{r3-h1}
Let $X < \Sym(\Omega) = \SSS_n$, where $(\Omega,X)$ is as listed above, that is, 
$d \geq 4$ when $X = PSL_d(q)$, $PSU_d(q)$, or $PSp_d(q)$, and 
$d \geq 5$ when $X = P\Omega^\pm_d(q)$. Suppose that $p =2 |n$. Then $h(X) \leq 3$ if 
$X = P\Omega^+_{d}(q)$ with $4|d \geq 8$, and $h(X) \leq 2$ otherwise. 
\end{lemma}
  
\begin{proof}
As mentioned in the proof of Lemma \ref{r3-orbits}, $2|n$ implies that $q$ is odd.
We follow the proof of Lemma \ref{h1} and its notation. First, $\dim H^2(X,\F)$, which is the $2$-rank of the Schur multiplier of $X$, is $\leq 2$ if $X = P\Omega^+_d(q)$ with $4|d \geq 8$, and 
$\leq 1$ otherwise by \cite[Theorem 5.1.4]{KL}. Hence it suffices to show that 
\begin{equation}\label{r3-hom}
   \dim \Hom(X_1,(\F,+)) \leq 1,
\end{equation}  
where $X_1$ is the point stabilizer in $X$ of a point in $\Omega$. Without loss we may 
replace $X$ by its central cover $SL_d(q)$, $SU_d(q)$, $Sp_d(q)$, or $\Omega^\pm_d(q)$.
We also fix a basis $(e_1,\ldots,e_d)$ of the natural module $W$ of $X$.

Consider the case $X = SL_d(q)$. Then $X_1 = \Stab_X(\la e_1,e_2 \ra) = Q \rtimes Y$,
where $|Q|$ is a $q$-power and $Y = (SL_2(q) \times SL_{d-2}(q)) \rtimes C_{q-1}$. Since 
$q$ is odd, we have that $O^2(Q) = Q$ and $O^{2}(SL_e(q)) = SL_e(q)$ for any $e \geq 2$. 
Hence (\ref{r3-hom}) follows.

From now on we may assume $X \neq SL_d(q)$. We can then choose $e_1$ to be
singular and let $X_1 = \Stab_X(\la e_1\ra)$. 

Let $X = SU_d(q)$. Then $X_1 = Q \rtimes Y$, where $|Q|$ is a $q$-power and 
$Y = SU_{d-2}(q) \rtimes C_{q^2-1}$. As $d \geq 4$ and $q$ is odd, we see that 
$O^2(SU_{d-2}(q)) = SU_{d-2}(q)$, and so (\ref{r3-hom}) follows.  

Suppose $X = Sp_d(q)$. Then $X_1 = Q \rtimes Y$, where $|Q|$ is a $q$-power and 
$Y = Sp_{d-2}(q) \rtimes C_{q-1}$. As $d \geq 4$ and $q$ is odd, we see that 
$O^2(Sp_{d-2}(q)) = Sp_{d-2}(q)$, yielding (\ref{r3-hom}).

Suppose $X = \Omega^\epsilon_d(q)$. Then $X_1 = Q \rtimes Y$, where $|Q|$ is a $q$-power and 
$Y = \Omega^\epsilon_{d-2}(q) \rtimes C_{q-1}$. As $d \geq 5$ and $q$ is odd, we see that 
$O^2(\Omega^\epsilon_{d-2}(q)) = \Omega^\epsilon_{d-2}(q)$, and so we are done.
\end{proof}   

\begin{proof}[Proof of Theorem \ref{r3-main}]
Assume the contrary: $\Res^{\AAA_n}_X V$ is irreducible; in particular, 
$V$ is irreducible. Note that $X$ is not $2$-transitive on $\Omega$. Hence, by the main results of \cite{KS1, KS2} we must have that $p = 2$ or $3$. 

\smallskip
(i) In the case $X = PSp_4(2)' \cong \AAA_6$, we have $n =  15$, and so 
$\dim V \geq 13$ (see \cite{ModAt}), whereas the largest dimension $\bfr(X)$ of irreducible 
$\F X$-modules is at most $10$, cf. \cite{JLPW}. Thus $V$ is reducible over $X$. 
Next, in the cases $X = SU_4(2) \cong PSp_4(3)$, respectively $SL_4(2)$, $Sp_6(2)$, we have
$n \geq 27$, $35$, $63$, and $\bfr(X) \leq 81$, $70$, $512$, respectively, according to \cite{JLPW}.
Certainly, $\Res^{\AAA_n}_X V$ is reducible if $\dim V > \bfr(X)$. So we must have 
that $\dim V \leq\bfr(X)$. Since $\bfr(X) < (n^2-5n+2)/2$,  by \cite[Lemma 6.1]{GT1} we see
that $V \cong D^{(n-1,1)}$ is isomorphic to the heart of the natural permutation module $\F\Omega$.
As in the proof of Lemma \ref{small}, we conclude that $X$ is $2$-transitive on $\Omega$,
a contradiction.

\smallskip    
(ii) We may now assume that $X$ is not isomorphic to any of the groups considered in (i). 
Direct computation shows that $2^{(n-8)/4} > |X|^{1/2}$.
Since $\Res^{\AAA_n}_X V$ is irreducible, we must have that 
\begin{equation}\label{r3-dim}
  \dim(V) < 2^{(n-8)/4},
\end{equation}  
which implies by Propositions \ref{ext1} and \ref{ext2} that $V$ extends to $\SSS_n$. Applying the Main Theorem and Theorem 3.10 of \cite{KS1}, we again arrive at the contradiction that $X$ is 
$2$-transitive in the case $p = 3$, as well as in the case $p = 2 \nmid n$. 

Thus we have shown that $p = 2|n$. The upper bound (\ref{r3-dim}) implies by
Theorem \ref{hom} that $d_3(V) > d_1(V)$. Also, $f_1(X) = 1$, $f_2(X) = 2$ by Lemma
\ref{e22}, and $e_3(X) \geq h(X)+1$ by Lemmas \ref{r3-orbits} and \ref{r3-h1}. Now
we can apply Theorem \ref{reduction}.
\end{proof}

Non-standard rank $3$ permutation representations of finite classical groups, as well
as other primitive subgroups of $\AAA_n$, will be considered elsewhere.
    

\end{document}